\documentclass[10pt,a4paper]{article}
\usepackage[utf8]{inputenc}
\usepackage[left=2.5cm, top=3cm, right=2.5cm, bottom=3cm]{geometry}                
 \usepackage{graphicx}
 \usepackage{amssymb, mathtools}
  \usepackage{times}
 \usepackage{amsthm,amsfonts,euscript,
 mathrsfs,graphics,color,amsmath,latexsym,marginnote}
 \usepackage{dsfont}

\numberwithin{equation}{section}

\newtheorem{theorem}{Theorem}[section]

\newtheorem{lemma}[theorem]{Lemma}
\newtheorem{proposition}{Proposition}

\newtheorem{question}{Question}
\theoremstyle{definition}
\newtheorem{definition}[theorem]{Definition}
\newtheorem{remark}{Remark}

\newcommand{\eps }{\epsilon}
\newcommand{\rr }{\rho}

\newcommand{\tc }{\mathtt{c}}

\newcommand{\R}{\mathbb{R}}

\newcommand{\T}{\mathbb{T}}
\newcommand{\C}{\mathbb{C}}

\newcommand{\Z}{\mathbb{Z}}
\newcommand{\N}{\mathbb{N}}
\newcommand{\cX}{\mathcal{X}}

\newcommand{\bral}{[\![}
\newcommand{\brar}{]\!]}

 \title{Sobolev norms explosion for the cubic NLS on irrational tori}

\author{Filippo Giuliani$^{1}$, Marcel Guardia$^{2, 3, 4}$ \thanks{The authors are supported by the European Research Council (ERC) 
under the European Union's Horizon 2020
research and innovation programme under grant agreement 
No 757802.}  \\
\small${}^{1}$ Dipartimento di Matematica, Politecnico di Milano, Milano, Italy.\\
\small${}^{2}$ Departament de Matem\`atiques, Universitat Polit\`ecnica de Catalunya (UPC), Barcelona, Spain.\\
\small${}^{3}$ IMTECH, Universitat Polit\`ecnica de Catalunya (UPC), Barcelona, Spain.\\
\small${}^{4}$ Centre de Recerca Matem\`atica, Barcelona, Spain.
}

\begin{document}
\maketitle

\begin{abstract} 
We consider the cubic nonlinear Schr\"odinger equation on $2$-dimensional irrational tori. We construct solutions which undergo growth of Sobolev norms. More concretely, for every $s>0$, $s\neq 1$ and almost every choice of spatial periods we construct solutions whose $H^s$ Sobolev norms grow by any prescribed factor. Moreover, for a set of spatial periods with positive Hausdorff dimension we construct solutions whose Sobolev norms go from arbitrarily small to arbitrarily large. We also provide estimates 
for the time needed to undergo the norm explosion.

Note that the irrationality of the space periods decouples the linear resonant interactions into products of $1$-dimensional resonances, reducing considerably the complexity of the resonant dynamics usually used to construct transfer of energy solutions. 

However, one can provide these growth of Sobolev norms solutions by using quasi-resonances relying on Diophantine approximation properties of the space periods.
\end{abstract}

\tableofcontents

\section{Introduction}

In the last decades there has been a lot of effort in understanding transfers of energy phenomena for linear and nonlinear PDEs on compact manifolds. 
A fundamental issue consists on constructing solutions mainly Fourier supported on a set of resonant modes that exchange energy among themselves as time evolves. 
Many works have been devoted to constructing different possible qualitative behaviors of such transfers. In \cite{GT}, \cite{GPT}, \cite{HausP17} the authors provide existence of periodic in time transfers of energy (usually called \emph{beating} effects) for nonlinear Schr\"odinger (NLS) equations on $\T:=\R/2\pi\Z$. More recently, in \cite{GGMP}, the authors prove the existence of solutions exhibiting \emph{chaotic-like} exchanges of energy for the cubic Wave and Beam equation on $\T^2:=\T\times \T$ (see also \cite{GGMPr}). We mention also \cite{GiulianiDC}, where the exchange of energy among couples of resonant modes relies on the presence of fast diffusion channels for the first order normal form of some nonlinear resonant PDEs on $\T$.

A fundamentally different question is to investigate whether it is possible to construct solutions exhibiting  transfer of energy  between modes of characteristically different scales. This phenomenon, known as \emph{forward} (backward) \emph{cascade} when the energy is transferred from low to high (high to low) modes, may lead to the \emph{growth} of higher order Sobolev norms as time evolves.

The importance of such phenomenon has been highlighted by Bourgain \cite{Bourgain00b}, who proposed the following question as one of the main problems for the XXI century  in the study of Hamiltonian PDEs.

\begin{question}Are there solutions $u(t)$ of the cubic defocusing nonlinear Schr\"odinger equation
\[\mathrm{i} u_t=\Delta u-|u|^2 u\qquad \text{ on }\qquad \T^2\]
such that
\[
\limsup_{t\to \infty} \| u(t)\|_{H^s(\T^2)}=\infty
\]
for some $s>1$?
\end{question}

Several works have been devoted to find polynomial upper bounds for the growth of high order  Sobolev norms (we mention for instance \cite{Bourgain96}, \cite{CatoireW10}, \cite{CollianderDKS01}, \cite{Sohinger11}, \cite{Staffilani97}, \cite{PTV17}). The results for \emph{lower bounds} on the growth are more scarce. The first ones are due to Bourgain \cite{Bou95}, \cite{Bourgain96} and Kuksin \cite{Kuksin96}, \cite{Kuksin97}, \cite{Kuksin97b}.

In $2010$, Colliander, Keel, Staffilani, Takaoka, Tao  \cite{CKSTT} provided the following outstanding result in the direction of the Bourgain conjecture.

\begin{theorem}\label{thm:Tao}
Let $s>1$ and $\mathcal{C}\gg 1$, $\mu\ll 1$. Then, there exists a solution $u(t)$ of the cubic NLS on $\T^2$ and $T>0$ such that
\[
\|u(0)\|_{H^s(\T^2)}\le \mu, \qquad \| u(T)\|_{H^s(\T^2)}\geq \mathcal{C}.
\]
\end{theorem}
The solutions given by Theorem \ref{thm:Tao} undergo an arbitrarily large, but finite, growth of Sobolev norms after certain time. These solutions follow closely the orbits of a \emph{resonant model} (called Toy model in \cite{CKSTT}) which is obtained by restricting the resonant Hamiltonian to the finite dimensional subspace of functions with a certain finite Fourier support $\Lambda\subset\Z^2$, that has to satisfy several combinatorial properties. The construction of the set $\Lambda$  and the analysis of the dynamics of the resonant model are crucial steps of the proof of Theorem \ref{thm:Tao}.

We observe that Theorem \ref{thm:Tao} can be seen as a result of Lyapunov instability in the $H^s$ topology for the origin $u=0$ of the cubic NLS, which is an elliptic fixed point. Thus we refer to orbits displaying the norms explosion, as the ones given in Theorem \ref{thm:Tao}, as \emph{unstable solutions}.

 In \cite{GuardiaK12} {(see also the erratum \cite{GuardiaK12Err})} the authors provided estimates for the instability  time $T$ by making a deep analysis of the dynamics of the finite dimensional resonant model . In particular they proved that in the case the initial $H^s$-norm is not assumed to be small then the instability time $T$ has a polynomial upper bound with respect to the growth $\mathcal{C}/\mu$
 \begin{equation*}
 T\lesssim (\mathcal{C}/\mu)^c \qquad \text{for some constant}\,\,c=c(s)>0,
\end{equation*}
 otherwise the upper bound is super exponential in the growth
 \begin{equation*}
 T\lesssim e^{(\mathcal{C}/\mu)^c} \qquad \text{for some constant}\,\,c=c(s)>0.
 \end{equation*}
 In \cite{Guardia14} the second author of the present paper proved the existence of solutions of the cubic NLS on $\T^2$ with convolution potential undergoing arbitrarily large growth of Sobolev norms.  Theorem \ref{thm:Tao} has been extended to the quintic NLS by Haus-Procesi \cite{HPquintic} and, later, to any NLS with analytic nonlinearity by Guardia-Haus-Procesi \cite{GuardiaHP16}.

 Concerning results of $H^s$-instability for different objects rather than elliptic fixed points, we mention Hani \cite{Hani12} and Guardia-Hani-Haus-Maspero-Procesi \cite{GuardiaHHMP19}. 
 
 {The aforementioned results concerns large but finite growth of Sobolev norms. Existence of solutions displaying an unbounded growth, in the spirit of the Bourgain's conjecture, have been proved for linear time-dependent equations \cite{Bourgain99}, \cite{Delort2014}, \cite{BGMR18}, \cite{Maspero18g}, \cite{FaouRaphael}, \cite{Liang}, \cite{Thomann2020GrowthOS}, \cite{MasperoDispersive}, and, for nonlinear PDEs, by Hani-Pausader-Tzvetkov-Visciglia \cite{HaniV2015} for the cubic NLS on the product space $\R\times \T^d$ and G\'erard-Grellier \cite{GerardG10}, \cite{GerardG11} for solutions of the Szeg\"o equation.}
 
 \medskip
 
As we have already said, the construction of unstable orbits is usually based on the study of the resonant dynamics of the PDE.
It is reasonable then to argue that the more complicated  the resonant structure of a system is, the richer should be its resonant dynamics. This is why in general we expect to be more difficult to appreciate instability phenomena for equations on $1$-dimensional spatial domains\footnote{Actually in the case of the cubic NLS the system is completely integrable and the Sobolev norms are controlled for all time.}.
In higher dimension a similar situation occurs when we consider \textbf{irrational tori}, namely 
\begin{equation*}
\T^2_{\lambda}:=\T_{\lambda_1}\times \T_{\lambda_2}, \qquad \lambda:=(\lambda_1, \lambda_2) \qquad \T_{\lambda_i}:=\R/2\pi \lambda^{-1}_i \Z,
\end{equation*}
where $\lambda\in \R^2$ is an irrational vector, i.e. $\lambda\cdot k\neq 0$ for all $k\in\Z^2\setminus\{0\}$. Indeed, as it was observed in  \cite{StafWil}, in the case in which $\lambda_1^2/\lambda_2^2$ is not an integer,
a resonant relation 
\[
\lambda_1^2 (j_1^2-j_2^2+j_3^2-j_4^2)+\lambda^2_2 (k_1^2+k_2^2-k_3^2+k_4^2)=0,
\]
decouples into two ``one-dimensional" resonant relations
\begin{equation*}
j_1^2-j_2^2+j_3^2-j_4^2=0\qquad \text{and}\qquad
k_1^2+k_2^2-k_3^2+k_4^2=0.
\end{equation*}
{Therefore there exist only the resonant quartets $(j_1, k_1), (j_1,k_2), (j_2,k_1), (j_2,k_2)$, that are vertices of rectangles with horizontal and vertical edges.}

\medskip

It is a general belief that the irrationality of the torus should mitigate the boundary effects and produce a sort of weak dispersion.
Recent works on the analysis of the spreading of energy to high modes and Sobolev stability results for PDEs on irrational tori  support this intuition. 

Recently Staffilani and collaborators started investigating the spreading of energy for solutions of the cubic NLS on {irrational tori}.
In \cite{StafWil} Staffilani-Wilson remark that it is not possible to apply the overall strategy of \cite{CKSTT} to obtain solutions undergoing growth of Sobolev norms. This is due to the lack of ``enough'' resonant quartets that are the building blocks for the construction of the aforementioned set $\Lambda$.
 Moreover they give a quantitative estimate of the energy spreading in the time scale of local theory. 
 In \cite{StafWil2} Hrabsky, Pan, Staffilani and Wilson refine the above analysis and perform numerical experiments. They are able to keep track in a good quantitative way on how far the bulk of the support of  solutions with initial conditions of any size may travel after an arbitrary fixed time.
 
Another evidence of obstructions to instability of NLS equations on irrational tori is provided by the works of Deng \cite{Deng} and Deng-Germain \cite{DengGerm}, where the authors prove that on $3$-dimensional irrational tori the polynomial in time upper bounds for the growth of Sobolev norms have a smaller degree compared to the rational case. 
We also mention \cite{BMflat} for polynomial time estimates on the growth of Sobolev norms of solutions of linear Schr\"odinger equation with time-dependent potential on irrational tori and recent works of long time stability in Sobolev spaces, based on normal form methods \cite{FeolaMont}, \cite{FIM}, \cite{BFGI}, \cite{BLM}.
 
 \medskip

 In this paper we prove the existence of solutions of the cubic NLS on (almost all) $2$-dimensional irrational tori undergoing an arbitrarily large (but finite) growth of their Sobolev norms.
The key idea is to apply the mechanism of \cite{CKSTT} for a \emph{quasi-resonant} model and to obtain such model normalizing just a finite number of terms of the Hamiltonian. This  avoids small divisor problems that appear in performing a Birkhoff normal form for NLS on irrational tori. Then, we use some arguments of Diophantine approximation to prove that the dynamics of our quasi-resonant model  approximates well the dynamics of the cubic NLS for a certain range of time.


\section{Main Results}

Let us consider the cubic NLS equation
\begin{equation}\label{eq:NLS}
\mathrm{i} u_t=\Delta u-|u|^2 u, \qquad u=u(t, x, y), \qquad (x, y)\in \T^2_{(1, \omega)}=\T\times\T_{\omega}.
\end{equation}
We consider the following Fourier expansion for functions $u\colon \T^2_{(1, \omega)}\to \C$
\[
u(x, y)=\sum_{n:=(j, k)\in \Z^2} a_{n}\,e^{\mathrm{i} (j x+\omega k y)}, \qquad a_{n}:=\frac{1}{2\pi \omega^{-1} }\int_{\T\times \T_{\omega}} u(x, y)\,e^{-\mathrm{i} (j x+\omega k y) }\,dx \,dy.
\]
The Hamiltonian of  equation \eqref{eq:NLS} with respect to the symplectic form $-\mathrm{i}\sum_{n\in \Z^2} d a_{n}\wedge d\overline{a_{n}}$ is given by
\[
H=H^{(2)}+H^{(4)}
\]
where
\begin{equation}\label{def:eigen}
\begin{split}
H^{(2)}&=\sum_{n\in \Z^2} \lambda(n) |a_{n}|^2\qquad  \qquad\text{with}\qquad \lambda(n)= |j|^2+\omega^{2} |k|^2
\\
H^{(4)}&=\sum_{n_1-n_2+n_3-n_4=0} a_{n_1}\,\overline{a_{n_2}}\,a_{n_3}\,\overline{a_{n_4}}.
\end{split}
\end{equation}
Let $s\geq 0$, we define the Sobolev spaces 
\[
H^s:=H^s(\T^2_{(1, \omega)}):=\left\{ u=\sum_{n:=(j, k)\in \Z^2} a_n\,e^{\mathrm{i} (j x+\omega k y)} : \| u \|_s^2=\sum_{n\in \Z^2} |a_n|^2 \langle n \rangle^{2s}<\infty \right\}
\]
where
\[
\langle n \rangle:=\max\{ 1, |n| \}.
\]


As in \cite{GuardiaK12} we provide two results, that differ from requiring or not the smallness of the norm of the initial conditions. In the case of arbitrarily small data solutions we see that the bounds on the time of instability worsen when we consider stronger Diophantine conditions on the tori lengths.
We restrict ourselves to $\omega\in [1,+\infty)$ since one can scale any  torus to obtain one with this property.

\begin{theorem}\label{thm:nosmall}
There exists a set of irrational numbers $\mathcal{S}\subset[1, +\infty)$, which has full measure and Hausdorff dimension 1, such that for all $\omega\in \mathcal{S}$ the following holds.

Let $s>0, s\neq 1$ and  fix $\mathcal{C}\gg 1$. Then the cubic NLS \eqref{eq:NLS} on the irrational torus $\T^2_{(1, \omega)}$ possesses a solution $u(t)$ such that
\begin{equation}\label{bound:lowerC}
\| u(T)\|_s\geq \mathcal{C} \| u(0)\|_s,
\end{equation}
where
\begin{equation}\label{time1}
T\le \exp\left(\,{\mathcal{C}^{\beta}}\right) 
\end{equation}
for some constant $\beta=\beta(s)>0$.  Moreover, there exists a constant $\eta=\eta(s)>0$ such that 
\[
 \|u(t)\|_{0}\leq \mathcal{C}^{-\eta}\quad \text{ for all }\quad t\in [0,T].
\]
\end{theorem}

%

Next theorem imposes that the initial $s$-Sobolev norm of the solution is arbitrarily small and explodes after certain (long) time.

\begin{theorem}\label{thm:small}
Fix $s>0, s\neq 1$. There exists a set $\widetilde{\mathcal{S}}_{2s}\subset[1, +\infty)$, which has  Hausdorff dimension $(1+s)^{-1}$, such that for all $\omega\in \widetilde{\mathcal{S}}_{2s}$ the following holds.

Fix $\mathcal{C}\gg 1$ and $\mu\ll1$. 
 Then, there exists a solution $u(t)$ of the cubic NLS \eqref{eq:NLS} on the irrational torus $\T^2_{(1, \omega)}$ and $T>0$ such that
\begin{equation}\label{bound:uppermu}
\| u(0)\|_s\le \mu, \qquad \| u(T)\|_s\geq \mathcal{C}.
\end{equation}
Moreover, for any $\tau>2s$, there exists a subset $\widetilde{\mathcal{S}}_\tau\subset \widetilde{\mathcal{S}}_{2s}$ with Hausdorff dimension $2(2+\tau)^{-1}$, such that if $\omega \in \widetilde{\mathcal{S}}_\tau$, the time $T$ satisfies
\begin{equation}\label{time2}
T\le\exp\left(\,\frac{1}{\tau-2s}\left(\frac{\mathcal{C}}{\mu}\right)^{\tilde\beta}\right) 
\end{equation}
for some constants $\tilde\beta(s)>0$.
\end{theorem}


Some comments are in order.
\begin{itemize}

\item The set $\mathcal{S}$ and the family of sets $\widetilde{\mathcal{S}}_\tau$ are defined through Diophantine approximation conditions of $\omega$ and are specified in Section \ref{def:Diophantine} below.

\item
Even if Theorem \ref{thm:nosmall} does not assume any smallness assumption on the $H^s$-norm of the initial datum, we are not able to provide polynomial upper bounds on the diffusion time as in \cite{GuardiaK12}. This comes from the fact that our \emph{Toy model} is not resonant, but just quasi-resonant. This is independent from the fact that we consider small or not the norm of the initial data.

\item Deng and Germain \cite{Deng}, \cite{DengGerm} showed that the growth of Sobolev norms for solutions of NLS is expected to be weaker on irrational tori. It is not easy to compare the time estimates that we obtain with these results because (i) our estimates are not optimal and are actually close to the ones obtained in the rational case, (ii) the upper bounds on the growth provided in \cite{Deng}, \cite{DengGerm} are for infinite time, while we are able to control the evolution for large, but finite, time.

\item Theorems \ref{thm:nosmall} and \ref{thm:small} can be seen as a transfer of energy counterpart to the ``energy localization'' result provided in \cite{StafWil2}. Indeed, \cite{StafWil2} (see also \cite{StafWil}) gives a control on the energy spreading from low to  high modes for equation \eqref{eq:NLS} on irrational tori. More precisely, the authors prove that the solutions which are supported at initial time at ``low'' Fourier modes ($\|(j,k)\|\leq N$ for some given $N$) can only excite ``very high modes'' ($\|(j,k)\|\geq M$ for a much larger $M$, depending on $N$ and the size of the initial Sobolev norm) after very long time.

The result in \cite{StafWil2} does not contradict ours for two reasons. First,  the time needed to attain growth in Theorems \ref{thm:nosmall} and \ref{thm:small} is much larger than the stability time in \cite{StafWil2}. Second, we attain growth of Sobolev norms by exciting high modes which are lower than the ones analyzed in  \cite{StafWil2}. 

The fine analysis of the time estimates and the size of the excited Fourier modes that we carry out allow us to suitably modify the approach in \cite{CKSTT} to obtain growth of Sobolev norms by drifting along the quasi-resonances of \eqref{eq:NLS}.

\item Our \emph{Toy model} is obtained through a partial normalization of the Hamiltonian of degree four. In particular we just need to normalize a finite number of terms in this Hamiltonian. This avoids small divisor problems.
 that may cause loss of derivatives and unboundness of the Birkhoff transformation. Indeed the irrationality of the torus implies that the combinations of linear frequencies accumulate to zero. This is in contrast with the case of the NLS on $\T^2$, where the linear eigenvalues are all integer numbers.


\end{itemize}

%

\subsection{Classical results on Diophantine approximation}\label{def:Diophantine}
A key point in the proofs of Theorems \ref{thm:nosmall} and \ref{thm:small} is to use Diophantine approximation properties for the length $\omega$ of the irrational torus. We devote this section to state several classical results in Diophantine approximation {(we mention \cite{BRV} for a modern survey on the subject)}. They will allow us to describe the sets $\mathcal{S}$ and $\widetilde{\mathcal{S}}_s$ introduced in the theorems.

\begin{definition}\label{def:psiapprox}
Let $\psi\colon \N\to \R^+:=[0, \infty)$ be a decreasing function. We say that $\omega\in \R$ is $\psi$-approximable if there exist infinitely many $(p, q)\in \Z\times\N$ such that
\[
\left|\omega-\frac{p}{q} \right|\le \frac{\psi(q)}{q}.
\]
We say that $(p, q)$ is a $\psi$-convergent of $\omega$.
\end{definition}

Note that the classical  Dirichlet approximation theorem states that all irrational numbers are $\psi$-approximable with $\psi(q)=q^{-1}$.

\begin{definition}\label{def:Wpsi}
Let $\psi\colon \N \to \R^+$ be a decreasing function.
Let us define
\[
W(\psi)=\left \{ \omega\in [1, \infty) : \omega\,\,\mathrm{\,\,is}\,\,\psi\,-\,\text{approximable}\right \}.
\]
\end{definition}
We denote by $m(\cdot)$ the Lebesgue measure on $\R$ and by $W(\psi)^c$ the complementary of $W(\psi)$ in $[1, \infty)$. We recall the following classical result in Diophantine approximation theory.
\begin{theorem}[Khinchin \cite{Kinch}]
Let $\psi\colon \N \to \R^+$ be a decreasing function. Then
\begin{align*}
&m(W(\psi))=
0 \qquad \qquad \mathrm{if}\qquad\,\sum_{q\in \N} \psi(q)<\infty,\\
&m(W(\psi)^c)=
0 \qquad \qquad \mathrm{if}\qquad\,\sum_{q\in \N} \psi(q)=\infty.
\end{align*}
\end{theorem}

For Theorem \ref{thm:nosmall}, we consider the function
\begin{equation}\label{choice:psi:1}
\psi(q)=\frac{\tc}{q\,\log q}
\end{equation}
for some $\tc\geq 1$.  We say that $\omega$ is $( \log, \tc)$-approximable if $\omega$ is $\psi$-approximable with $\psi$ as in \eqref{choice:psi:1}. The set $\mathcal{S}$ considered in this theorem is just the set of $(\log, \tc)$-approximable numbers for some  $\tc\geq 1$. 
By Khinchin theorem the set of $(\log, \tc)$-approximable numbers has full Lebesgue measure and maximal Hausdorff dimension.

For Theorem \ref{thm:small} we have to be more restrictive. We consider the function 
\begin{equation}\label{choice:psi:2}
\psi(q)=\frac{\tc}{q^{1+\tau}}
\end{equation}
for some $\tc\geq 1, \tau>0$. We say that $\omega$ is $( \tau, \tc)$-approximable if it is $\psi$-approximable with $\psi$ as in \eqref{choice:psi:2}.
 Then, we define the set  $\widetilde{\mathcal{S}}_\tau$ introduced in Theorem \ref{thm:small} as  the set of $\omega\in [1,+\infty)$ which are $(\tau', \tc)$-approximable numbers for some $\tau'>\tau$ and $\tc\geq 1$ and \emph{are not} $\tilde\psi$-approximable with 
\[
 \tilde\psi(q)=\frac{1}{q^{\log q}}.
\]
Note that the set of $\tilde\psi$-approximable numbers is contained in the set of Liouville numbers, hence it has zero Lebesgue measure.
By Kinchin theorem also the set $\widetilde{\mathcal{S}}_s$ has measure zero. Regarding the Hausdorff dimension of such sets we have the following classical result.
\begin{theorem}[Jarn\'ik-Besicovitch \cite{Jarnik}, \cite{Besi}]
Let $\tau>0$, then the Hausdorff dimension of the set of $(\tau, 1)$-approximable numbers is $2(\tau+2)^{-1}$.
\end{theorem}
Therefore the set of $\tilde\psi$-approximable numbers  has 0 Hausdorff dimension, while the set $\widetilde{\mathcal{S}}_s$ has Hausdorff dimension $(1+s)^{-1}$. 
%
%

\subsection{Heuristics}
Now we give the main ideas on how we construct the unstable solutions of Theorems \ref{thm:nosmall} and \ref{thm:small}. These solutions have the following form
\[
u(t, x, y)=v(t, x, y)+R(t, x, y), \qquad v(t, x, y):=\sum_{n\in \Lambda} a_n(t)\,e^{\mathrm{i} (j x+\omega k y) },
\]
where $\Lambda\subset \Z^2$ is a finite set and $R(t)$ is a function which is small, at least for some time, in the Banach space $\ell^1$ (see the definition in \eqref{def:ell1}), which has a weaker topology with respect to the Sobolev spaces $H^s$\footnote{ We remark that this implies that the closeness in $\ell^1$ does not imply closeness in $H^s$.}. The compactly Fourier supported function $v(t)$ is the solution of a particular truncation of the NLS Hamiltonian. Then it is an approximate solution of the full cubic NLS.

More precisely the function $v(t)$ is solution of the system obtained by restricting the Hamiltonian terms of degree $4$, after a normalization procedure, to the finite dimensional subspace
\begin{equation*}
\mathcal{V}_{\Lambda}:=\{ a_n=0,\,\,n\notin \Lambda \},
\end{equation*}
which is generated by setting all the modes out of $\Lambda$ at rest. The normalization procedure consists in eliminating all the monomials of degree $4$ with exactly one $a_n$ with $n\notin \Lambda$ (see \eqref{def:eigen}). Thanks to the conservation of momentum it is easy to see that there is a finite number of such terms. More precisely the Hamiltonian is transformed as
\[
H^{(2)}+H^{(4)} \qquad \rightarrow \qquad H^{(2)}+\mathcal{N}+\mathcal{R},
\]
where $\mathcal{N}$ is the partially normalized Hamiltonian of degree $4$ and $\mathcal{R}$ is a function of order $6$ at the origin.
Then
 $\mathcal{V}_{\Lambda}$ is invariant under the flow of the truncated normalized Hamiltonian $\mathcal{N}$.   

  We refer to the system obtained by the restriction on $\mathcal{V}_{\Lambda}$ as our \emph{quasi-resonant model}, that will describe the effective dynamics of the modes in $\Lambda$. This is the counterpart of the Toy model of \cite{CKSTT}. A crucial difference is that the set $\Lambda$  in our case is not made by resonant quartets, but only by modes that are close to be resonant. More precisely it will be resonant with respect to the quadratic Hamiltonian
 \begin{equation}\label{def:lambdatilde}
 \sum_{n\in \Z^2} \tilde{\lambda}(n) |a_n|^2, \qquad \tilde{\lambda}(n):=j^2+\frac{p^2}{q^2} k^2,
 \end{equation}
 where $(p, q)\in \Z^2$ is a suitable $\psi$-convergent of $\omega$ (see Definition \ref{def:psiapprox}).
   The closeness to resonances of the cubic NLS on $\T^2_{(1, \omega)}$ is measured by the quality of the approximation of $\omega$ by the rational number $p/q$. 

 The properties of the set $\Lambda$ determine the equations of the quasi-resonant model. Our set $\Lambda$  has the same combinatorial properties as the set $\Lambda$  in \cite{CKSTT}. Then the equations that describe its evolution are the same.

We know that for this system there exists an orbit $r(t)$ displaying a transfer of energy from modes in $\Lambda$ with characteristically different scale. Since (in rotating coordinates) the vector field of the quasi-resonant model is homogenous of degree $3$ we can consider rescaled solutions
\[
r^{\lambda}(t):=\lambda^{-1} r(\lambda^{-2} t).
\]
The next step consists in proving an approximation argument, namely showing that solutions of the cubic NLS that starts close enough (in the weak norm) to $r^{\lambda}(0)$ shadow the orbit $r^{\lambda}(t)$ for a certain range of times. The terms that we have to control in order to guarantee that the solutions arising from a neighborhood of $r^{\lambda}(0)$ do not diverge too much are mainly two: (i) the remainder $\mathcal{R}$ coming from the normal form procedure, (ii) the error that is originated by considering the solution of a quasi-resonant Hamiltonian system, instead of a resonant one.

The control on the remainder $\mathcal{R}$ is guaranteed by taking the scaling factor $\lambda$ large enough, that means to consider solutions with very small (in $\ell^1$-norm) initial data. This is because the remainder is a function of higher order (six) at the origin.

The second term to control, after passing to rotating coordinates (to eliminate the linear part of the equation), presents a phase-term
\[
\exp\big({\mathrm{i}\, \Omega_{\omega}(n_1, \dots, n_4) t}\big)-1\quad \sim \quad   \Omega_{\omega}(n_1, \dots, n_4)\,t ,
\]
with
\[
 \Omega_{\omega}(n_1, \dots, n_4):=\lambda(n_1)-\lambda(n_2)+\lambda(n_3)-\lambda(n_4).
\]
This term has to be controlled over long time and it reflects how bad the quasi-resonant model approximates the full PDE. To control this term we need some good upper bounds for the resonant combinations $\Omega_{\omega}(n_1, \dots, n_4)$. Recalling \eqref{def:eigen} and \eqref{def:lambdatilde}, this shall be done by proving that $|\Omega_{\omega}-\Omega_{p/q}|$, where $\Omega_{p/q}$ is defined as $\Omega_{\omega}$ by replacing $\lambda_{n_i}$ with $\tilde{\lambda}_{n_i}$, is of the order of $q\, \psi(q)$. Thanks to the choices \eqref{choice:psi:1}, \eqref{choice:psi:2} this function is decreasing in $q$, hence we can provide better bounds by choosing larger $q$. We remark that at this step the rate of $\psi(q)$ does not play any role.

After the approximation argument we have to ensure that solutions $u(t)$ close to $r^{\lambda}(t)$ satisfies \eqref{bound:lowerC} and \eqref{bound:uppermu}. This is done by using the fact that the Birkhoff normalizing transformation is close to the identity and the set $\Lambda$ satisfies certain properties.

 Eventually, to impose that the $H^s$-Sobolev norm of the initial datum is small, as in Theorem \ref{thm:small}, we need to consider $\omega$ as a $(\tau, \tc)$-approximable number with $\tau>2s$. We also require that $\omega$ is  not Liouville to have estimates for the convergents of $\omega$.

\medskip

\paragraph{Plan of the paper} In Section \ref{sec:bnf} we prove an abstract result of partial normalization of the NLS Hamiltonian. 
In Section \ref{sec:Lambda} we construct the set $\Lambda$  and we show that we can apply the partial normalization of Section \ref{sec:bnf} with respect to $\mathcal{V}_{\Lambda}$. In Section \ref{sec:resmodel} we study the dynamics of the quasi-resonant model. In Section \ref{sec:approxarg} we prove the approximation argument. Finally, in Section \ref{sec:conclusion}, we prove the bounds on the Sobolev norms at initial and final time and we conclude the proofs of Theorems \ref{thm:nosmall} and \ref{thm:small}.

\paragraph{Acknowledgements} The authors are supported by the European Research Council (ERC) 
under the European Union's Horizon 2020
research and innovation programme (grant agreement 
No. 757802). M. Guardia is also supported by the Catalan Institution for Research and Advanced
Studies via an ICREA Academia Prize 2019. This work is also supported by the Spanish State Research Agency, through the Severo Ochoa and Mar\'ia de Maeztu Program for Centers and Units of Excellence in R\&D (CEX2020-001084-M)

\section{Birkhoff normal form}\label{sec:bnf}

\subsection{Notations}
We use the following notation:
\begin{itemize}
\item $a\sim b$ if there exist two constants $C_1, C_2>0$ which may depend on $\omega$ such that $C_1 a\le b \le C_2 a$.
\item $a\lesssim b$ if there exists a constant $C=C(\omega)>0$ such that $a\le C b$.
\item We denote by $n=(j, k)$ the generic element of $\Z^2$.
\item Given a set $\Lambda\subset \Z^2$ we denote by $|\Lambda|$ the cardinality of $\Lambda$.
\item Given a Hamiltonian $H$ we denote by $X_H$ its vector field.
\end{itemize}

\subsection{The Hamiltonian structure}


The equation \eqref{eq:NLS} can be seen as an infinite dimensional system of ODEs for the Fourier coefficients
\begin{equation}\label{eq:ode}
-\mathrm{i} \dot{a}_{n}=\lambda(n) a_{n}+\sum_{n_1-n_2+n_3=n} a_{n_1}\,\overline{a_{n_2}}\,a_{n_3}, \qquad n\in\Z^2.
\end{equation}
For our purpose it is useful to remove some cubic terms by using the Gauge invariance of equation \eqref{eq:NLS}
\[
a_{n}=\rr_{n}\,e^{\mathrm{i} G t}.
\]
Choosing $G=2 \| u\|^2_{L^2}$,  equation \eqref{eq:ode} becomes
\begin{equation}\label{eq:simpleNLS}
-\mathrm{i} \dot{\rr}_{n}={\lambda(n) \rho_n}-|\rr_{n}|^2 \rr_{n}+\sum_{\substack{n_1-n_2+n_3=n,\\ n_1\neq n_2}} \rr_{n_1}\,\overline{\rr_{n_2}}\,\rr_{n_3}, \qquad n\in \Z^2
\end{equation}
and its Hamiltonian is given by
\begin{equation}\label{def:calH}
\mathcal{H}=H^{(2)}+\mathcal{H}^{(4)}
\end{equation}
\[
\mathcal{H}^{(4)}:=-\frac{1}{4} \sum_{n\in \Z^2} |\rr_{n}|^4+\frac{1}{4} \sum_{\substack{n_1-n_2+n_3-n_4=0,\\ n_1\neq n_2, n_4}} \rr_{n_1}\,\overline{\rr_{n_2}}\,\rr_{n_3}\,\overline{\rr_{n_4}}.
\]
Now we partially normalize the above Hamiltonian. The normalization consists in eliminating the monomials of degree $4$ which are supported on a quartet $(n_1, n_2, n_3, n_4)$ with exactly one $n_i$ out of a given set $\Lambda$. In the following we introduce some definitions which are useful to identify such Hamiltonian terms. 
 
%
%

 \medskip

We consider a finite subset $\Lambda\subset \Z^2$ and we denote by $\mathcal{H}^{(4, d)}$, $0\le d\le 4$, the terms Fourier supported on 
\begin{equation}\label{def:A}
\begin{aligned}
\mathcal{A}(d):=\Big\{ & (n_1, n_2, n_3, n_4)\in (\Z^2)^4 : n_1-n_2+n_3-n_4=0,\\
\,\, &\text{exactly}\,\, d\,\, \text{integer vectors $n_i$ belong to}\,\, \Z^2\setminus \Lambda \Big\}.
\end{aligned}
\end{equation}
In particular $\mathcal{H}^{(4, 0)}$ is Fourier supported just on modes in $\Lambda$.
Given $r\in \R$, we call
\begin{equation}\label{def:Omega}
\Omega_{r}(n_1, \dots, n_4):=\sum_{i=1}^4 (-1)^{i+1} |j_i|^2+r^{2} \sum_{i=1}^4 (-1)^{i+1} |k_i|^2.
\end{equation}
These are the resonant combinations of order $4$ of the Laplacian on the torus $\T^2_{(1, r)}$.

In the Birkhoff normal form procedure and in the approximation argument we shall take into account respectively the lower and upper bounds of such combinations.  
Let us denote 
\begin{equation}\label{def:UL}
\mathcal{L}_k:=\inf_{\mathcal{A}(k)} |\Omega_{\omega}(n_1, \dots, n_4)|, \qquad \mathcal{U}_k:=\sup _{\mathcal{A}(k)} |\Omega_{\omega}(n_1, \dots, n_4)|.
\end{equation}

\medskip

We shall perform the normal form procedure in the space
\begin{equation}\label{def:ell1}
\ell^1:=\left\{ \rr\colon \Z^2\to \mathbb{C} : \| \rr\|_{\ell^1}:= \sum_{n\in \Z^2} |\rr_{n}|<\infty  \right\}.
\end{equation}
We recall that this space is an algebra with the convolution product. For $\eta>0$ we denote by 
\[
B(\eta):=\left\{  \rr\colon \Z^2\to \mathbb{C} : \| \rr\|_{\ell^1}<\eta \right\}.
\]
We shall use the following classical lemma.
\begin{lemma}\label{lem:young}
Let 
\[
F=\sum_{\substack{\sum_{i=1}^{d+1} \sigma_i n_i=0,\\ \sigma_i\in\{\pm\}}} F^{\sigma_1 \dots \sigma_{d+1}}_{n_1 \dots n_{d+1}}\,\rr^{\sigma_1}_{n_1}\dots \rr^{\sigma_{d+1}}_{n_{d+1}}, \qquad \rr_n^{+}:=\rr_n, \quad \rr_n^{-}:=\overline{\rr_n}
\]
 be a homogenous Hamiltonian of degree $d+1$ preserving momentum. Then
\[
\| X_F(\rr) \|_{\ell^1}\lesssim \bral F \brar \| \rr\|_{\ell^1}^d,
\]
where
\[
\bral F \brar:=\sup_{(\sigma_i, n_i)} |F^{\sigma_1 \dots \sigma_{d+1}}_{n_1 \dots n_{d+1}}|.
\]
Moreover, let $G$ be a homogenous, momentum preserving, Hamiltonian of degree $d'+1$ of the same form of $F$. Then $\{ F, G \}$ is a homogenous, momentum preserving, Hamiltonian of degree $d+d'$ and
\[
\bral \{F, G\} \brar\lesssim \bral F \brar\,\bral G \brar.
\]
{
We deduce that
\begin{equation}\label{ineq:lie}
\| [X_F, X_G](\rho) \|_{\ell_1}\lesssim \bral F \brar\,\bral G \brar \| \rho\|^{d+d'-1}_{\ell^1},
\end{equation}
where $[\cdot, \cdot ]$ denotes the standard Lie bracket between vector fields.
}

\end{lemma}
\begin{proof}
The proof relies on Young's inequality and the algebra property of $\ell^1$. We refer for instance to the Appendix A in \cite{GT} for more details. {The last inequality comes from the fact that $X_{\{ F, G\}}=[X_F, X_G]$.}
\end{proof}
The main result of this section is the following.
\begin{proposition}[Weak Birkhoff Normal Form]\label{prop:wbnf}
Let $\Lambda\subset \Z^2$ be a finite set such that 
\begin{equation}\label{assump:wbnf}
\Omega_{\omega}(n_1, n_2, n_3, n_4)\neq 0 \qquad \forall (n_1, n_2, n_3, n_4)\in \mathcal{A}(1),
\end{equation}
where $\Omega_{\omega}$ and $\mathcal{A}(1)$ are defined respectively in \eqref{def:Omega} and \eqref{def:A}.

Then, there exist $\eta>0$ and a symplectic change of coordinates $\rr=\Gamma(\beta)\colon B(\eta)\to B(2\eta)$ which transforms the Hamiltonian $\mathcal{H}$ in \eqref{def:calH} into the Hamiltonian
\begin{equation}\label{def:Htilde}
\widetilde{\mathcal{H}}:=\mathcal{H}\circ \Gamma=H^{(2)}+\mathcal{H}^{(4, 0)}+\mathcal{H}^{(4, \geq 2)}+\mathcal{R}.
\end{equation}
Moreover 
the following estimates hold 
\begin{align}
\sup_{\beta\in B(\eta)}\| \Gamma^{\pm 1}(\beta)- \beta \|_{\ell^1} &\lesssim \mathcal{L}_1^{-1}\, \eta^3, \label{bound:GammaId} \\ \label{bound:remainder}
\sup_{\beta\in B(\eta)} \| X_{\mathcal{R}}(\beta)\|_{\ell^1} &\lesssim \mathcal{L}_1^{-1}\, \eta^5+\mathcal{L}_1^{-2} \eta^7,
\end{align}
where $\mathcal{L}_1$ is defined in \eqref{def:UL}.
The estimate \eqref{bound:GammaId} holds also for the inverse $\Gamma^{-1}$.
\end{proposition}

We remark that the smallness condition on the radius $\eta$ that we need to impose to ensure that $\Phi^t$ maps $B(\eta)$ to $B(2\eta)$ (and similar for its differential $D\Phi^{t}$) is of the form
\[
\eta^2 \mathcal{L}_1^{-1}\ll 1.
\]
We will verify this condition for our particular choice of set $\Lambda$  at the end of Section \ref{sec:Lambda}.

\begin{proof}[Proof of Proposition \ref{prop:wbnf}]
We consider the following homogenous Hamiltonian
\[
F(\beta):=\sum_{\substack{n_1-n_2+n_3-n_4=0}} F^{+-+-}_{n_1 \dots n_4}\,\,  \beta_{n_1}\,\overline{\beta_{n_2}}\,\beta_{n_3}\,\overline{\beta_{n_4}}
\]
with
\begin{equation}\label{def:F}
F^{+-+-}_{n_1,\dots, n_4}:=\begin{cases}
\dfrac{1}{4 \mathrm{i} \Omega_{\omega}(n_1, \dots, n_4)}, \qquad (n_1, \dots, n_4)\in \mathcal{A}(1),\\[4mm]
0 \qquad \qquad \qquad\qquad \quad \,\,\,  \text{otherwise}.
\end{cases}
\end{equation}
By the assumption in \eqref{assump:wbnf}, the denominator $\Omega_{\omega}(n_1, \dots, n_4)$ never vanishes if $(n_1, \dots, n_4)\in \mathcal{A}(1)$. By conservation of momentum and the fact that $\Lambda$ is finite the set $\mathcal{A}(1)$ is finite, which implies that the constant $\mathcal{L}_1$ introduced in \eqref{def:UL} is a positive constant  depending on $\Lambda$ . Therefore,  $F$ is well defined and the associated Cauchy problem  is a well-posed ODE. Moreover, it is easy to see that with this choice of $F$ we have that
\begin{equation}\label{homo}
\{  F, H^{(2)}  \}+\mathcal{H}^{(4, 1)}=0.
\end{equation}
By the fact that $X_F$ is homogenous and analytic, its flow $\Phi^t$ maps smoothly $B(\eta)$ to $B(2\eta)$ for all $0\le t\le 1$ provided that $\eta$ is small enough. We consider the time-one flow map $\Phi^t|_{t=1}=:\Gamma$.
By \eqref{def:F}, the definitions in \eqref{def:UL} and  Lemma \ref{lem:young}, we have that
\[
\| X_F(\beta) \|_{\ell^1}\lesssim \mathcal{L}_1^{-1} \| \beta\|_{\ell^1}^3.
\]
Then, by the mean value theorem, we obtain \eqref{bound:GammaId} for $\Gamma$.
{Since $\Gamma$ is the time-one flow map of an autonomous system, the same argument can be used to prove the above inequality for the inverse $\Gamma^{-1}$.}
The new Hamiltonian is obtained by Taylor expanding $\mathcal{H}\circ \Phi^t$ at $t=0$
\begin{align*}
\mathcal{H}\circ \Gamma =&\,\mathcal{H}+\{ F, \mathcal{H} \}+\frac{1}{2}\int_0^1 (1-t) \{ F, \{ F, \mathcal{H}\}\}\circ \Phi^t\,dt\\
=&\,H^{(2)}+\mathcal{H}^{(4, 0)}+\mathcal{H}^{(4, \geq 2)}+\{ F, H^{(4)}\}\\
&\,+\frac{1}{2}\int_0^1 (1-t) \{ F, \{ F, {H^{(2)}}\}\}\circ \Phi^t\,dt
+\frac{1}{2}\int_0^1 (1-t) \{ F, \{ F, \mathcal{H}^{(4)}\}\}\circ \Phi^t\,dt\\
=&\,H^{(2)}+\mathcal{H}^{(4, 0)}+\mathcal{H}^{(4, \geq 2)}+\mathcal{R},
\end{align*}
where, using \eqref{homo},
\[
\mathcal{R}:=\{ F, H^{(4)}\}-\frac{1}{2}\int_0^1 (1-t) \{ F, \mathcal{H}^{(4, 1)}\}\circ \Phi^t\,dt+\frac{1}{2}\int_0^1 (1-t) \{ F, \{ F, \mathcal{H}^{(4)}\}\}\circ \Phi^t\,dt.
\]
By a straightforward computation, we have that
\[
X_{\mathcal{R}}=[X_F, X_{H^{(4)}}]-\frac{1}{2}\int_0^1 (1-t) D \Phi^{-t} [X_F, X_{\mathcal{H}^{(4, 1)}}]\circ \Phi^t\,dt+\frac{1}{2}\int_0^1 (1-t) D \Phi^{-t} [X_F,  [X_F, X_{\mathcal{H}^{(4)}}]]\circ \Phi^t\,dt,
\]
where $[\cdot, \cdot ]$ denotes the standard Lie bracket between vector fields. By the fact that $\Phi^t\colon B(\eta)\to B(2\eta)$ for $0\le t\le 1$ and $X_F$ is a (bounded) homogenous polynomial one can prove that 
\[
\sup_{\beta\in B(\eta)} \| D\Phi^{-t}(\beta) [\hat{\beta}]\|_{\ell^1}\lesssim \|\hat{\beta}\|_{\ell^1} \qquad \forall t\in (0, 1),
\]
provided that $\eta$ is small enough.
Then, by Lemma \ref{lem:young}-\eqref{ineq:lie} and the fact that $\Phi^t\colon B(\eta)\to B(2\eta)$ for $0\le t\le 1$ we obtain the bound \eqref{bound:remainder}.
This concludes the proof.
\end{proof}

\section{Construction of the set $\Lambda$}\label{sec:Lambda}
In this section we construct a set $\Lambda$ for which the Hamiltonian $\mathcal{H}^{(4, 0)}$ restricted to $\mathcal{V}_{\Lambda}$ gives the same equations of the Toy model considered in \cite{CKSTT}. Later we show that this set $\Lambda$ satisfies the assumption of the Proposition \ref{prop:wbnf}.

\medskip

The  set  $\Lambda$ that we consider is a suitable scaling of the set constructed in \cite{CKSTT}. It has to satisfy certain combinatorial properties which simplify the analysis of the dynamics of the modes in $\Lambda$. First we recall the construction and the properties of the set considered in \cite{CKSTT}, which we denote by $\tilde{\Lambda}$.
Following \cite{CKSTT}, the set $\tilde{\Lambda}$ can be decomposed as the disjoint union of generations $\tilde{\Lambda}_i$ as
\[
\tilde{\Lambda}=\tilde{\Lambda}_1\cup\dots\cup \tilde{\Lambda}_N.
\]
We say that a quartet $(\tilde{n}_1, \tilde{n}_2, \tilde{n}_3, \tilde{n}_4)$ is a \emph{nuclear family} if $\tilde{n}_1, \tilde{n}_3\in \tilde{\Lambda}_i$ and $\tilde{n}_2, \tilde{n}_4\in \tilde{\Lambda}_{i+1}$ for some $i=1, \dots, N-1$ and they form a non-degenerate rectangle\footnote{Note that the rectangles in $\Z^2$ give the resonant quartets for \eqref{eq:NLS} in the torus $\T^2_{(1,1)}=\T\times\T$}. We shall first construct a set with the following properties, which were considered in \cite{CKSTT}:
\begin{itemize}
\item[$(P1)$] (Closure): If $\tilde{n}_1, \tilde{n}_2, \tilde{n}_3\in\tilde{\Lambda}$ are three vertices of a rectangle then the fourth vertex belongs to $\tilde{\Lambda}$ too.
\item[$(P2)$] (Existence and uniqueness of spouse and children): For each $1\le i\le N-1$  and every $\tilde{n}_1\in\tilde{\Lambda}_i$ there exists a unique spouse $\tilde{n}_3\in \tilde{\Lambda}_i$ and unique (up to trivial permutations) children $\tilde{n}_2, \tilde{n}_4\in\tilde{\Lambda}_{i+1}$ such that $(\tilde{n}_1, \tilde{n}_2, \tilde{n}_3, \tilde{n}_4)$ is a nuclear family in $\tilde{\Lambda}$.
\item[$(P3)$] (Existence and uniqueness of parents and sibling): For each $1\le i\le N-1$  and every $\tilde{n}_2\in\tilde{\Lambda}_{i+1}$ there exists a unique sibling $\tilde{n}_4\in\tilde{\Lambda}_{i+1}$ and unique (up to trivial permutations) parents $\tilde{n}_1, \tilde{n}_3\in\tilde{\Lambda}_{i}$ such that $(\tilde{n}_1, \tilde{n}_2, \tilde{n}_3, \tilde{n}_4)$ is a nuclear family in $\tilde{\Lambda}$.
\item[$(P4)$] (Non-degeneracy): A sibling of any mode $m$ is never equal to its spouse.
\item[$(P5)$] (Faithfulness): Apart from nuclear families, $\tilde{\Lambda}$ contains no other rectangles.
{ \item[$(P6)$] If four points in $\tilde{\Lambda}$ satisfy $\tilde{n}_1-\tilde{n}_2+\tilde{n}_3-\tilde{n}_4=0$ then either the relation is trivial or such points form a family.}
\end{itemize}
{
 \begin{remark}
 The conditions $(P1)$--$(P5)$ are the ones considered in \cite{CKSTT}. The property $(P6)$ was also considered in \cite{GuardiaHHMP19}.
 \end{remark}
 }
 
The construction of the $\tilde{\Lambda}$ set is given by the following theorem proved in \cite{CKSTT} and \cite{GuardiaHHMP19}.

\begin{theorem}\label{thm:Iteam}
Fix any $\tilde{\eta}>0$ small. Then, there exists $\alpha\gg 1$  such that for any $N\gg 1$,  there exists a set  $\tilde{\Lambda}\subset \Z^2$ with
\[
\tilde{\Lambda}:=\tilde{\Lambda}_1\cup\dots\cup \tilde{\Lambda}_N
\]
which satisfies $(P1)$-$(P6)$ and also
\begin{equation*}
\frac{\sum_{n\in\Lambda_{N-1}}|\tilde{n}|^{2s}}{\sum_{n\in\Lambda_3} |\tilde{n}|^{2s}}\geq \frac{1}{2}\,2^{(s-1)(N-4)}.
\end{equation*}
Moreover there exists $R=R(N)$ satisfying
\[
 e^{\alpha^N}\leq R \leq e^{2(1+\tilde{\eta})\alpha^N}
\]
and $C>0$ independent of $N$ such that
\[
C^{-1} R\le  |\tilde{n}|\le C\,3^N\, R 
\qquad 
\forall \tilde{n}\in \tilde{\Lambda}_i, \qquad i=1, \dots, N
\]
%
and
\begin{equation*}
\frac{\sum_{\tilde{n}\in\Lambda_{j}}|\tilde{n}|^{2s}}{\sum_{\tilde{n}\in\Lambda_i} |\tilde{n}|^{2s}}\lesssim\,e^{sN}
\end{equation*}
for any $1\le i< j\le N$.
\end{theorem}

Our set $\Lambda$  is obtained by scaling each $\tilde{n}=(\tilde{j}, \tilde{k})\in\tilde{\Lambda}$ in the following way
\begin{equation}\label{rescaling}
\tilde{j} \mapsto p\,\tilde{j}, \qquad \tilde{k}\mapsto q\, \tilde{k}
\end{equation}
where $(q, p)\in \Z^2$ is a suitable $\psi$-approximation of $\omega$ to be chosen later. The key difference with \cite{CKSTT} is that our rescaling is anisotropic.

We have that $\Lambda$ is the disjoint union
\[
\Lambda=\Lambda_1\cup\dots\cup\Lambda_N,
\] 
where the generations $\Lambda_i$ correspond to the image of $\tilde{\Lambda}_i$ through the scaling \eqref{rescaling}.

We point out that the nuclear families in $\tilde{\Lambda}$ are not mapped into rectangles through the anisotropic scaling \eqref{rescaling}. Then, we have to ``translate'' the properties $(P1)$ and $(P5)$. {Since $(P2)$-$(P3)$-$(P4)$ are purely combinatorial and $(P6)$ involves just linear relations}, they remain unchanged under scaling. First we introduce the following definition.

\begin{definition}\label{def:pqfamily}
A  $(p, q)$-\emph{family} is a non-degenerate parallelogram with vertices $n_1, n_2, n_3, n_4\in\Lambda$ which satisfy
\begin{equation}\label{respq}
\begin{aligned}
&n_1-n_2+n_3-n_4=0,\\
&\Omega_{p/q}(n_1, n_2, n_3, n_4)=({j}_1^2-{j}_2^2+{j}_3^2-{j}_4^2)+\frac{p^2}{q^2} ({k}_1^2-{k}_2^2+{k}_3^3-{k}_4^2)=0.
\end{aligned}
\end{equation}
\end{definition}

\begin{lemma}
The image of a nuclear family in $\tilde{\Lambda}$ under the scaling \eqref{rescaling} is a $(p, q)$-family. Moreover, if $(n_1, n_2, n_3, n_4)$ is a $(p, q)$-family in $\Lambda$ then it is the image of a nuclear family. In particular $n_1, n_3\in \Lambda_i$, $n_2, n_4\in \Lambda_{i+1}$ for some $i=1, \dots, N-1$.
\end{lemma}
\begin{proof}
 After the scaling \eqref{rescaling} a nuclear family $(\tilde{n}_1, \tilde{n}_2, \tilde{n}_3, \tilde{n}_4)$ in $\tilde{\Lambda}$ is mapped into a $(p, q)$-family in $\Lambda$. Indeed $n_i=(j_i, k_i)$ is such that $j_i=q \tilde{j}_i$, $k_i=p\tilde{k}_i$, then
\begin{equation}\label{eq:fampqfam}
\Omega_{p/q}(n_1, n_2, n_3, n_4)=p^2 \Omega_1(\tilde{n}_1, \tilde{n}_2, \tilde{n}_3, \tilde{n}_4)=0.
\end{equation}
We observe that the momentum relation is preserved by \eqref{rescaling}. The fact that a $(p, q)$-family in $\Lambda$ is the image of a family comes again from \eqref{eq:fampqfam} and the definition of $\Lambda$. The last assertion comes from the definition of $\Lambda_i$.
 \end{proof}

By the previous lemma we can replace the properties $(P1)$ and $(P5)$ with the following:
\begin{itemize}
\item[$(P1')$] (Closure): If $n_1, n_2, n_3\in\Lambda$ are three vertices of a non-degenerate parallelogram satisfying \eqref{respq} then the fourth vertex belongs to $\Lambda$ too.
\item[$(P5')$] (Faithfulness): Apart from $(p, q)$-families, $\Lambda$ contains no other parallelograms satisfying \eqref{respq}.
\end{itemize}

%

By the scaling \eqref{rescaling}, Theorem \ref{thm:Iteam} can be translated into the following.

\begin{theorem}\label{thm:gen}
Fix any $\tilde{\eta}>0$ small. Then, there exists $\alpha\gg 1$ such that for any $N\gg 1$ and any $p,q\in\mathbb{N}$, there exists a set $\Lambda\subset p\Z\times q \Z\subset \Z^2$ with
\[
\Lambda:=\Lambda_1\cup\dots\cup\Lambda_N,
\]
which satisfies conditions $(P1'), (P2),(P3),(P4),(P5'), (P6)$ and also
\begin{equation*}
\frac{\sum_{n\in\Lambda_{N-1}}|n|^{2s}}{\sum_{n\in\Lambda_3} |n|^{2s}}\gtrsim \,2^{(s-1)(N-4)}.
\end{equation*}
Moreover, we can ensure that each generation $\Lambda_i$ has $2^{N-1}$ disjoint frequencies satisfying 
\begin{equation}\label{bound:gen}
\frac{\sum_{n\in\Lambda_{j}}|n|^{2s}}{\sum_{n\in\Lambda_i} |n|^{2s}}\lesssim\,e^{sN},
\end{equation}
for any $1\le i< j\le N$, and 
\begin{equation}\label{def:R}
C^{-1}\,q\, R\le  |{n}|\le C\,q\,3^N\, R, \qquad \forall {n}\in {\Lambda}_i, \qquad i=1, \dots, N,
\end{equation}
where $C>0$ is independent of $N$ and
 $R=R(N)$ is the constant introduced in Theorem \ref{thm:Iteam}, which satisfies
\[
 e^{\alpha^N}\leq R \leq e^{2(1+\tilde{\eta})\alpha^N}.
 \]
\end{theorem}

{\begin{remark}
We note that the first two inequalities come from the fact that $p\sim q$ and, for $n=(p\,j, q\,k)\in p\Z\times q\Z$, $\tilde{n}=(j, k)$
\[
C_1\, q^2\, |\tilde{n}|^2\le \min\{ p^2, q^2\} |\tilde{n}|^2 \le |n|^{2}\le \max\{ p^2, q^2 \} |\tilde{n}|^2\le C_2\, q^2\, |\tilde{n}|^2
\]
for some constants $C_1, C_2>0$.
\end{remark}}


Now we have to ensure that we can apply Proposition \ref{prop:wbnf} with the  set $\Lambda$ that we have constructed. The main point is to verify that $\Lambda$ satisfies the assumption \eqref{assump:wbnf}. The property $(P1')$ implies that (recall the definition of $\mathcal{A}(1)$ in  \eqref{def:A})
\[
\Omega_{p/q}(n_1, n_2, n_3, n_4)\neq 0 \qquad \forall (n_1, n_2, n_3, n_4)\in \mathcal{A}(1).
\]
We claim that this implies also
\[
\Omega_{\omega}(n_1, n_2, n_3, n_4)\neq 0 \qquad \forall (n_1, n_2, n_3, n_4)\in \mathcal{A}(1).
\]
More precisely we have the following stronger estimate.

\begin{lemma}\label{lem:L1} 
Let $(p, q)$ be a $\psi$-convergent of $\omega$ and assume that
\begin{equation}\label{assumption}
3^{2N}\,R^2\, \frac{\psi(q)}{q}\ll 1.
\end{equation}
Let $n_1, n_2, n_3\in \Lambda$, $n_4:=n_1-n_2+n_3\notin \Lambda$. Then
\[
|\Omega_{\omega}(n_1, n_2, n_3, n_4)|\gtrsim q^2.
\]
\end{lemma}
\begin{proof}
Let us define
\[
B:=\begin{pmatrix}
p & 0\\
0 & q
\end{pmatrix}, \qquad A:=\begin{pmatrix}
1 & 0\\
0 & \omega^2
\end{pmatrix}.
\]
and let us consider $\tilde{n}_1, \tilde{n}_2, \tilde{n}_3\in \tilde{\Lambda}\subset \Z^2$ the pre-images of $n_1, n_2, n_3$ through the scaling \eqref{rescaling}, namely $n_i=B \tilde{n}_i$, $i=1, 2, 3$.  Thus the pre-image of $n_4$ is $\tilde{n}_4=\tilde{n}_1-\tilde{n}_2+\tilde{n}_3\in\Z^2$.
We have that
\begin{equation}\label{a}
\begin{aligned}
\Omega_{1}(\tilde{n}_1, \tilde{n}_2, \tilde{n}_3, \tilde{n}_4)&=|\tilde{n}_1|^2-|\tilde{n}_2|^2+|\tilde{n}_3^2|-|\tilde{n}_1-\tilde{n}_2+\tilde{n}_3|^2\\
&=2\,\langle \tilde{n}_1-\tilde{n}_2, \tilde{n}_2-\tilde{n}_3 \rangle.
\end{aligned}
\end{equation}
 By construction of the $\tilde{\Lambda}$ set (recall the property of closure $(P1)$)
\begin{equation}\label{b}
|\Omega_{1}(\tilde{n}_1, \tilde{n}_2, \tilde{n}_3, \tilde{n}_4)|\geq 1
\end{equation}
since $\Omega_{1}(\tilde{n}_1, \tilde{n}_2, \tilde{n}_3, \tilde{n}_4)$ is a non-zero integer. This implies that
\begin{equation}\label{osserv}
|\langle \tilde{n}_1-\tilde{n}_2, \tilde{n}_2-\tilde{n}_3 \rangle|\geq \frac{1}{2}.
\end{equation}
We observe that
\begin{align*}
|\Omega_{\omega}(n_1, n_2, n_3, n_4)| &=2|\langle n_1-n_2, A\,(n_2-n_3)  \rangle|=2 |\langle \tilde{n}_1-\tilde{n}_2, (B^T A B) (\tilde{n}_2-\tilde{n}_3) \rangle|
\end{align*}
with
\[
B^T A B=\begin{pmatrix}
p^2 & 0\\
0 & q^2 \omega^2
\end{pmatrix}.
\]
Moreover,
\[
B^T A B=p^2 \left( \mathrm{I}+J\right), \qquad J:=\begin{pmatrix}
0 & 0\\
0 & r
\end{pmatrix}
\]
where $\mathrm{I}$ is the identity matrix and, recalling that $p\sim q$, $r$ satisfies
\[
|r|\lesssim \frac{\psi(q)}{q}.
\]
We have then
\[
|\langle \tilde{n}_1-\tilde{n}_2, (B^T A B) (\tilde{n}_2-\tilde{n}_3) \rangle|\geq p^2 \left( |\langle \tilde{n}_1-\tilde{n}_2, \tilde{n}_2-\tilde{n}_3 \rangle|-|\langle \tilde{n}_1-\tilde{n}_2, J( \tilde{n}_2-\tilde{n}_3) \rangle|  \right)
\]
and, using also that $|\tilde{n}_i|\lesssim 3^N R$, $i=1\ldots 4$,
\[
|\langle \tilde{n}_1-\tilde{n}_2, J( \tilde{n}_2-\tilde{n}_3) \rangle| \lesssim |r|\, 3^{2N} R^2\lesssim 3^{2N}\,R^2\, \frac{\psi(q)}{q}.
\]
By \eqref{osserv} and the assumption \eqref{assumption},   we have that
\[
|\langle \tilde{n}_1-\tilde{n}_2, (B^T A B) (\tilde{n}_2-\tilde{n}_3) \rangle|\geq \frac{p^2}{4} |\langle \tilde{n}_1-\tilde{n}_2, \tilde{n}_2-\tilde{n}_3 \rangle|.
\]
By \eqref{a}, \eqref{b}, we conclude that
\[
|\Omega_{\omega}(n_1, n_2, n_3, n_4)|\geq \frac{p^2}{2} |\Omega_1(\tilde{n}_1, \tilde{n}_2, \tilde{n}_3, \tilde{n}_4)|\geq \frac{p^2}{2}\gtrsim q^2 .
\]
\end{proof}

Therefore, we can apply Proposition \ref{prop:wbnf} considering the set $\Lambda$ given by Theorem \ref{thm:gen} provided that \eqref{assumption} holds. We shall choose later a suitable $q$ that satisfies this condition (recall that $\psi$ is decreasing, see Definition \ref{def:psiapprox}).

Hence we have a change of coordinates $\Gamma$ that puts the Hamiltonian $\mathcal{H}$ in \eqref{def:calH} in normal form
\[
\mathcal{H}\circ \Gamma=H^{(2)}+\mathcal{H}^{(4, 0)}+\mathcal{H}^{(4, \geq 2)}+\mathcal{R},
\]
where $\mathcal{R}$ shall be considered as a small remainder (see \eqref{bound:GammaId2} below).
\begin{remark}
We point out that $\mathcal{H}^{(4, 0)}$ is NOT a resonant Hamiltonian, namely $\{ H^{(2)}, \mathcal{H}^{(4, 0)}\}\neq 0$. However it commutes with the Hamiltonian \eqref{def:lambdatilde}.
\end{remark}

By Lemma \ref{lem:L1} and the fact that $\Lambda$ is finite we have that
\begin{equation}\label{bound:lowL1}
\mathcal{L}_1\gtrsim q^2.
\end{equation}
Then, the bounds \eqref{bound:GammaId} on the Birkhoff map and \eqref{bound:remainder} on the remainder of the Birkhoff procedure are 
\begin{equation}\label{bound:GammaId2}
\begin{split}
\sup_{\beta\in B(\eta)}\| \Gamma(\beta)- \beta \|_{\ell^1}&\lesssim q^{-2}\, \eta^3\\
\sup_{\beta\in B(\eta)} \| X_{\mathcal{R}}(\beta)\|_{\ell^1}&\lesssim q^{-2}\, \eta^5+q^{-4} \, \eta^7,
\end{split}
\end{equation}
where $\eta>0$ is given by Proposition \ref{prop:wbnf}.

\begin{remark}\label{rem:eta}
We recall that the smallness of $\eta$ is given by a relation like $\eta^2 \mathcal{L}_1^{-1}\ll 1$. By \eqref{bound:lowL1} and the fact that $q\gg 1$ we can consider $\eta>0$ as a small universal constant.
\end{remark}

\section{Quasi-resonant model}\label{sec:resmodel}
In this section we construct an orbit of the quasi-resonant model that displays the desired energy exchange behavior.
We introduce the rotating coordinates (recall \eqref{def:eigen})
\begin{equation}\label{def:rotcoord}
\beta_{n}=r_{n}\,e^{\mathrm{i} \lambda(n)\,t}.
\end{equation}
The Hamiltonian vector field associated to the Hamiltonian \eqref{def:Htilde} expressed in  rotating coordinates is
\begin{equation*}
\cX:= X_{\mathtt{H}^{(4, 0)}}+ X_{\mathtt{H}^{(4, \geq 2)}}+X_{\mathcal{R}'}
 \end{equation*}
with (recall \eqref{def:A})
\begin{align*}
&\mathtt{H}^{(4, 0)}:=\mathcal{H}^{(4, 0)}(\{ r_{n}\,e^{\mathrm{i} \lambda(n)t} \}_{n\in\Lambda}),\\
&\mathtt{H}^{(4, \geq 2)}:=\mathcal{H}^{(4, \geq 2)}(\{ r_{n}\,e^{\mathrm{i} \lambda(n)t} \}_{n\in\Z^2}),\\
&\mathcal{R}'(t)=\mathcal{R} (\{ r_{n}\,e^{\mathrm{i} \lambda(n)t} \}_{n\in\Z^2}).
\end{align*}
Defining
\begin{equation}\label{def:N}
\mathcal{N}:=\mathcal{H}^{(4, 0)}+\mathcal{H}^{(4, \geq 2)}\qquad \text{ and }\qquad {\mathcal{Q}^{(4, d)}}:=\mathtt{H}^{(4, d)}-\mathcal{H}^{(4, d)}, \qquad d=0, 2, 3, 4,
\end{equation}
we can write $\cX$ as 
\begin{equation}\label{vecfrot}
\cX=X_{\mathcal{N}}+X_{\mathcal{Q}^{(4, 0)}}+X_{\mathcal{Q}^{(4, \geq 2)}}+X_{\mathcal{R}'}.
\end{equation}
The Hamiltonian $\mathcal{N}$ gives the first order of the Hamiltonian in rotating coordinates. Since $\mathcal{H}^{(4, 0)}$ only possesses monomials suported in $\Lambda$ and $ \mathcal{H}^{(4, \geq 2)}$ only has monomials in $\mathcal{A}(2)$, $\mathcal{A}(3)$ and $\mathcal{A}(4)$ (see \eqref{def:A}), the subspace
\[
\mathcal{V}_{\Lambda}:=\{ r\colon \Z^2\to\C : r_{n}=0, \,\,\, n\notin \Lambda  \}
\]
is invariant. We analyze the dynamics of the Hamiltonian system defined by $\mathcal{N}$ on this subspace.

\begin{remark}\label{rem}
Observe that the vector fields $X_{\mathcal{H}^{(4, \geq 2)}}$ and $X_{\mathtt{H}^{(4, \geq 2)}}$ vanish on $\mathcal{V}_{\Lambda}$. 
\end{remark}
Reasoning as in \cite{CKSTT}, one can easily see that, for a set $\Lambda$ with the properties $(P1'), (P2), (P3), (P4), (P5'), (P6)$, the equation for $\mathcal{N}$ restricted to the subspace $\mathcal{V}_{\Lambda}$ is given by
\begin{equation}\label{eq:spouse}
-\mathrm{i} \dot{r}_n=-r_n |r_n|^2+2 r_{n_{\mathrm{child}_1}}\,r_{n_{\mathrm{child}_2}} \,\overline{r_{n_{\mathrm{spouse}}}} +2 r_{n_{\mathrm{parent}_1}}\,r_{n_{\mathrm{parent}_2}}\overline{r_{n_{\mathrm{sibling}}}},
\end{equation}
where the ``parental relations'' refer to the $(p,q)$-families in $\Lambda$ (see Definition \ref{def:pqfamily}).

\begin{lemma}[Intragenerational equality \cite{CKSTT}]
Consider the subspace
\[
\widetilde{\mathcal{V}_{\Lambda}}:=\left\{ r\in \mathcal{V}_{\Lambda} : r_{n}=r_{n'}\,\,\,\forall n, n'\in \Lambda_j\,\,\text{for some}\,\,j\in \{ 1, \dots, N \} \right\}
\]
where all the members of a generation take the same value. Then $\widetilde{\mathcal{V}_{\Lambda}}$ is invariant under the flow of \eqref{eq:spouse}.
\end{lemma}
To determine the restriction of \eqref{eq:spouse} to $\widetilde{\mathcal{V}_{\Lambda}}$ we set
\begin{equation*}
b_i=r_n \qquad \text{for any}\qquad n\in \Lambda_i.
\end{equation*}
The equation \eqref{eq:spouse} restricted on $\widetilde{\mathcal{V}_{\Lambda}}$ reads as
\begin{equation}\label{eq:toy}
\dot{b}_i=-\mathrm{i} b_i^2 \overline{b_i}+2\mathrm{i} \overline{b_i} (b^2_{i-1}+b^2_{i+1}), \qquad i=1, \dots, N.
\end{equation}
The following result has been proved in \cite{GuardiaK12}.
\begin{theorem}\label{thm:orbit}
Fix a large $\gamma\gg 1$. Then for any large enough $N$ and $\delta=\exp(-\gamma N)$, there exists a trajectory $b(t)$ of the system \eqref{eq:toy}, $\sigma>0$ independent of $\gamma$ and $N$ and $T_0>0$ such that
\begin{align*}
|b_3(0)|>1-\delta^{\sigma},& \qquad |b_j(0)|<\delta^{\sigma} \quad\,\,\, \mathrm{for}\,\,j\neq 3,\\
|b_{N-1}(T_0)|>1-\delta^{\sigma},& \qquad |b_j(T_0)|<\delta^{\sigma} \quad  \mathrm{for}\,\,j\neq N-1.
\end{align*}
Moreover there exists a constant $\mathbb{K}>0$ independent of $N$ such that $T_0$ satisfies
\begin{equation}\label{bound:timeT0}
0<T_0\leq  \mathbb{K}\,N\,\log(\delta^{-1})=\mathbb{K} \gamma N^2.
\end{equation}
\end{theorem}

 Since $\mathcal{N}$ in \eqref{def:N} is a homogenous Hamiltonian of degree $4$ we can consider the scaled solution
\begin{equation}\label{lambda}
b^{\lambda}(t)=\lambda^{-1} b(\lambda^{-2} t),
\end{equation}
where $b(t)$ is the trajectory given by Theorem \ref{thm:orbit}.
%
%
Then
\begin{equation}\label{def:rlambda}
r^{\lambda}(t, x, y)=\sum_{n\in \Z^2} r^{\lambda}_n(t)\,e^{\mathrm{i} (j x+\omega k y)}, \qquad r^{\lambda}_n(t):=\begin{cases}
b_i^{\lambda}(t) \qquad n\in\Lambda_i, \quad i=1, \dots, N,\\
0 \qquad \,\, \,\, \quad n\notin \Lambda
\end{cases}
\end{equation}
is a solution of $\mathcal{N}$ in \eqref{def:N}.
The life-span of $r^{\lambda}(t)$ is 
\begin{equation}\label{def:timeT}
T:=\lambda^{2} T_0\sim \lambda^2 \gamma N^2.
\end{equation}
 By Theorem \ref{thm:orbit}, arguing as in Lemma $9.2$ in \cite{GuardiaHP16}, we have
\begin{equation}\label{bound:base}
\sup_{t\in [0, T]} \| r^{\lambda}(t) \|_{\ell^1}\lesssim |\Lambda| \lambda^{-1}, \qquad |\Lambda|=N 2^{N-1}.
\end{equation}

\section{Approximation argument}\label{sec:approxarg}
In this section we prove that there exists a solution of \eqref{vecfrot} that stays close to $r^{\lambda}(t)$ in the $\ell^1$-topology for time $t\in [0, T]$.
To this end, we give an upper bound on the resonant combinations among modes in $\Lambda$. This will be useful in estimating the error generated by considering a quasi-resonant model instead of a resonant one.
\begin{lemma}\label{lem:U0}
 Assume that $\omega$ is $\psi$-approximable.
Let $(n_1, n_2, n_3, n_4)$ be a $(p, q)$-family (recall Definition \ref{def:pqfamily}). Then
\[
\left| j_1^2-j_2^2+j_3^2-j_4^2+\omega^2 (k_1^2-k_2^2+k_3^2-k_4^2) \right| \lesssim  \,  3^{2N} R^2 q\psi(q).
\]
Therefore, the constant $\mathcal{U}_0$ in \eqref{def:UL} satisfies
\[
\mathcal{U}_0\lesssim  \, 3^{2N} R^2q\psi(q).
\]
\end{lemma}
\begin{proof}
By definition of $(p, q)$-family we have $\Omega_{p/q}(n_1, n_2, n_3, n_4)=0$. Then,
\begin{align*}
\left|\Omega_{\omega}(n_1, n_2, n_3, n_4)\right|&=\left| \Omega_{\omega}(n_1, n_2, n_3, n_4)-\Omega_{p/q}(n_1, n_2, n_3, n_4)\right|\\
&=\left( \omega^2-\frac{p^2}{q^2} \right) \left| k_1^2-k_2^2+k_3^2-k_4^2 \right|\lesssim  \max_{n\in\Lambda}\{ |n|^2 \}\,\frac{\psi(q)}{q}\,.
\end{align*}
Hence, it is enough to use the bound \eqref{def:R}. The estimate on $\mathcal{U}_0$ comes from the fact that $\Lambda$ is a finite set.
\end{proof}

Now we show that, under appropriate conditions on the parameters involved in the proof, that is $N, q$ in \eqref{def:psiapprox} and $\lambda$ in \eqref{lambda}, there are solutions of \eqref{vecfrot} that shadow, in the $\ell^1$-norm, the trajectory $r^{\lambda}$ on the time interval $[0, T]$. We shall rely on the fact that our choices of the function $\psi$ in \eqref{choice:psi:1},\eqref{choice:psi:2} are such that $q\, \psi(q)$ is decreasing. 

This approximation argument is performed in a ball centered at the origin of $\ell^1$ with radius $O( \| r^{\lambda}(0)\|_{\ell^1})$.
Hence, to guarantee that we work in the ball $B(\eta)$, with $\eta$ given by Proposition \ref{prop:wbnf} and Remark \ref{rem:eta}, where we can apply Proposition \ref{prop:wbnf}, we need to impose that
\[
\lambda^{-1} N 2^N\ll 1.
\]
This will be implied (for $N$ large enough) by the first condition in \eqref{cond:final} required by the following proposition.


\begin{proposition}[Approximation argument]\label{prop:approxarg}
Fix $\epsilon>0$ small. Then, for $N\gg 1$ and $\lambda, q$ such that
\begin{equation}\label{cond:final}
\lambda\geq \exp(5^N),\qquad q\,\psi(q)\le N^{-7} 72^{-N} \lambda^{-2(1+\epsilon)} R^{-2}
\end{equation}
 the following holds.
 
If $r(t)$ is a solution of \eqref{vecfrot} such that
\begin{equation}\label{assump:initial}
\| r(0)-r^{\lambda}(0) \|_{\ell^1}\le \lambda^{-(1+2\epsilon)}
\end{equation}
then
\begin{equation}\label{bound:close}
\sup_{t\in [0, T]}\| r(t)-r^{\lambda}(t) \|_{\ell^1}\le \lambda^{-(1+\epsilon)}
\end{equation}
where $T$ is in \eqref{def:timeT}.
\end{proposition}
\begin{proof}
Note that \eqref{cond:final} implies \eqref{assumption}. Therefore, we can apply Lemma \ref{lem:L1}  and consider the estimates \eqref{bound:GammaId2}

We define $\xi(t):=r(t)-r^{\lambda}(t)$. The equation for $\xi$ can be written as $\dot{\xi}=Z_0+Z_1+Z_2+Z_3+Z_4$
where
\begin{align*}
&Z_0:=X_{\mathcal{R}'}(r^{\lambda}+\xi),\\
&Z_1:=D X_{\mathcal{N}}(r^{\lambda})[\xi],\\
&Z_2:=X_{\mathcal{N}}(r^{\lambda}+\xi)-X_{\mathcal{N}}(r^{\lambda})-D 
X_{\mathcal{N}}(r^{\lambda})[\xi],\\
&Z_3:=X_{\mathcal{Q}^{(4, 0)}}(r^{\lambda})+X_{\mathcal{Q}^{(4, \geq 2)}}(r^{\lambda}),\\
&Z_4:=X_{\mathcal{Q}^{(4, 0)}}(r^{\lambda}+\xi)-X_{\mathcal{Q}^{(4, 0)}}(r^{\lambda})+X_{\mathcal{Q}^{(4, \geq 2)}}(r^{\lambda}+\xi)-X_{\mathcal{Q}^{(4, \geq 2)}}(r^{\lambda}).
\end{align*}
Applying the $\ell^1$ norm to this equation we obtain
\[
\frac{d}{d t} \| \xi \|_{\ell^1}\le \|Z_0\|_{\ell^1}+\|Z_1\|_{\ell^1}+\|Z_2\|_{\ell^1}+\|Z_3\|_{\ell^1}+\|Z_4\|_{\ell^1}.
\]
We assume temporarily a \emph{bootstrap assumption}. We call $T^*$ to the 
supremum of the $T'>0$ such that\footnote{If this condition is true for all 
$T'>0$, one can just take $T_*=+\infty$.}
\[
\| \xi(t)\|_{\ell^1}\le 2   \lambda^{-(1+\epsilon)} \qquad \forall t\in [0, 
T'].
\]
Note that \eqref{assump:initial} implies $T_*>0$. A posteriori we will prove 
that  $T_*>T$, where $T$ is the time introduced in \eqref{def:timeT} and therefore  we 
can drop the bootstrap assumption.

 First we need a priori estimates on the terms $Z_i$ defined above.  We shall 
use repeatedly Lemma \ref{lem:young} and the bootstrap assumption without 
mentioning them. We remark that the change to rotating coordinates 
\eqref{def:rotcoord} does not affect the bounds.

\medskip

\noindent\textbf{Bound for $Z_0$.} By \eqref{bound:base} and the bootstrap assumption
\[
\left\|r^\lambda+\xi\right\|_{\ell^1}\lesssim |\Lambda| 
\lambda^{-1}+\lambda^{-(1+\epsilon)}\lesssim N 2^N \lambda^{-1}.
\]
Hence by \eqref{bound:GammaId2}
\[
\| Z_0 \|_{\ell^1}\lesssim  q^{-2} N^5 2^{5N} 
\lambda^{-5}.
\]
\noindent\textbf{Bound for $Z_1$.} By \eqref{bound:base} and considering that $\mathcal{N}$ is a homogenous Hamiltonian of degree $4$ we have 
\[
\| Z_1 \|_{\ell^1}\lesssim |\Lambda|^2 \lambda^{-2} \| \xi \|_{\ell^1}\lesssim 
N^2\, 4^{N}\,\lambda^{-2} \| \xi \|_{\ell^1}.
\]

\noindent\textbf{Bound for $Z_2$.} By 
\eqref{bound:base}, we have
\[
\| Z_2 \|_{\ell^1}\lesssim 
\lambda^{-1}|\Lambda| \| \xi\|_{\ell^1}^2\lesssim 
\lambda^{-2-\eps}\,|\Lambda|\, \| 
\xi\|_{\ell^1}\lesssim N 2\,^{N}\, \lambda^{-2-\eps}\, \| \xi \|_{\ell^1}.
\]
\noindent\textbf{Bound for $Z_3$.} By Remark \ref{rem}, $X_{\mathcal{Q}^{(4, 
\geq 2)}}(r^{\lambda})=0$ and so $Z_3=X_{\mathcal{Q}^{(4, 0)}}(r^{\lambda})$.
{Using that $\mathcal{Q}^{(4, 0)}$ is a homogenous Hamiltonian of degree $4$, we can reason as in Lemma \ref{lem:young} and bound $Z_3$, up to constants, by the product of $\| r^{\lambda} \|_{\ell^1}^3$ and the maximum of the coefficients of the Hamiltonian $\mathcal{Q}^{(4, 0)}$, that is
\[
\sup_{\substack{n_1, n_2, n_3, n_4\in \Lambda,\\
n_1-n_2+n_3-n_4=0,\\
\Omega_{p/q}(n_1, n_2, n_3, n_4)=0}} |\exp(\mathrm{i} \Omega_{\omega}(n_1, n_2, n_3, n_4) t)-1|\le  \sup_{\substack{n_1, n_2, n_3, n_4\in \Lambda,\\
n_1-n_2+n_3-n_4=0,\\
\Omega_{p/q}(n_1, n_2, n_3, n_4)=0}} | \Omega_{\omega}(n_1, n_2, n_3, n_4) |\,|t|.
\] 
Note that the supremum is taken over the $(p, q)$-families of $\Lambda$.
}
 Then, by Lemma \ref{lem:U0} and \eqref{bound:base}, we have
\[
\| Z_3(t) \|_{\ell^1}\lesssim \mathcal{U}_0 |\Lambda|^3 \lambda^{-3}\,t\lesssim 
 3^{2N} R^2 q\psi(q) |\Lambda|^3 \lambda^{-3}\,t
 \lesssim  
 N^3 72^N\,q\,\psi(q)\, \lambda^{-3}\,R^2\,t.
\]

\noindent\textbf{Bound for $Z_4$.} By \eqref{bound:base} we have
\[
\| Z_4 \|_{\ell^1}\lesssim  |\Lambda|^2 \lambda^{-2} \| \xi \|_{\ell^1}\lesssim 
N^2 4^{N} \lambda^{-2} \| \xi \|_{\ell^1},
\]
where we have used that $|e^{\mathrm{i}\alpha}-1|\le 2$ for all $\alpha\in \R$.

\medskip

By collecting the previous estimates we have
\[
\frac{d}{dt} \| {\xi}(t) \|_{\ell^1}\lesssim N^2 4^{N} \lambda^{-2} \| \xi(t) 
\|_{\ell^1} + N^3 72^{N}\,q\,\psi(q)\, \lambda^{-3}\,R^2\,t+q^{-2} N^5 
2^{5N} \lambda^{-5}.
\]
Then, by Gronwall Lemma and condition  
\eqref{cond:final},
\begin{align*}
\| \xi(t) \|_{\ell^1}
&\le C \left(\lambda^{-(1+2\epsilon)}+ \, N^3 72^{N}\,q\,\psi(q)\, 
\lambda^{-3}\,R^2\,t^2+q^{-2} N^5 2^{5N} \lambda^{-5}\, t   \right)\,\exp 
\left(C\, N^2\,4^{N}\,\lambda^{-2} t \right)
\end{align*}
for some universal constant $C>0$.
 Then, for $t\in [0, T]$, where $T$ is defined in \eqref{def:timeT}, one has  
\begin{align*}
\| \xi(t) \|_{\ell^1}
&\le C \left(\lambda^{-(1+2\epsilon)}+ \, N^7 72^{N}\,q\,\psi(q)\, 
\lambda\,R^2+q^{-2} N^7 2^{5N} \lambda^{-3}\right)\,\exp 
\left(C\, N^4\,4^{N}\right).
\end{align*}
 Taking $N\gg 1$ and assuming \eqref{cond:final} we have that
 \[
\| \xi( t) \|_{\ell^1}\le  \lambda^{-(1+\epsilon)} \qquad \forall t\in [0, T].
\]
Therefore $T_*>T$ and  we can drop the bootstrap assumption. This concludes the proof.
\end{proof}

\section{Conclusion of the proofs of Theorems \ref{thm:nosmall} and \ref{thm:small}}\label{sec:conclusion}

We first construct solutions that undergo growth of Sobolev norms. Then we complete the proof of Theorem \ref{thm:small} by assuming that the initial condition has small Sobolev norm.
We prove these theorems under the assumption $s>1$. At the end, we briefly mention how it can be easily adapted for $s\in (0,1)$.

\begin{lemma}\label{lem:ratio}
Fix $s>1$. Then, for $N\gg 1$ and assuming \eqref{cond:final} the following holds.
There exists a solution $\rho(t)$ of \eqref{eq:simpleNLS} such that
\[
\frac{\| \rho(T) \|_s^2}{\| \rho(0) \|_s^2}\gtrsim 2^{(s-1) (N-6)},
\]
where $T$ is the time introduced in \eqref{def:timeT}.
\end{lemma}
\begin{proof}
We consider the solution $\rho(t)$ of \eqref{eq:simpleNLS} such that $\rho(0)=r^{\lambda}(0)$, where $r^{\lambda}(t)$ has been defined in \eqref{def:rlambda}. We define
\[
S_i:=\sum_{n\in \Lambda_i} |n|^{2 s}, \qquad i=1, \dots, N.
\]
We start by giving a lower bound on the norm at time $T$ in terms of $S_{N-2}$
\[
\| \rho(T)\|_{s}^2\geq \sum_{n\in\Lambda_{N-2}} |n|^{2s} |\rho_n(T)|^2\geq S_{N-2}\,\inf_{n\in\Lambda_{N-2}} |\rho_n(T)|^2.
\]
Now we obtain a lower bound for $|\rho_n(T)|$ with $n\in \Lambda_{N-2}$. Using the change of variables $\Gamma$ obtained in Proposition \ref{prop:wbnf} and the rotating coordinates \eqref{def:rotcoord}, we can write the solution $\rho(t)$ as
\[
\rho(t)=\Gamma\left(\{ r_n(t)\,e^{\mathrm{i} \lambda(n) t} \} \right),
\]
where $r(t)$ is solution of the system \eqref{vecfrot} (the normalized equation in rotating coordinates). Now we show that $r(t)$ fits into the assumption of Proposition \ref{prop:approxarg} and so it is well approximated by the trajectory $r^{\lambda}$ given in \eqref{def:rlambda}.
To this end, we need to check condition \eqref{assump:initial}.
 Since $\rho(0)=r^{\lambda}(0)$ we have
\begin{equation}\label{bound:diff}
\| r(0)-r^{\lambda}(0) \|_{\ell^1}=\| r(0)-\rho(0) \|_{\ell^1}=\| r(0)-\Gamma(r(0)) \|_{\ell^1}. 
\end{equation}
Again by using Proposition \ref{prop:wbnf}, the estimate \eqref{bound:base}, the fact that $\rho(0)=r^{\lambda}(0)$ and \eqref{cond:final} we have
\begin{equation*}
\| r(0) \|_{\ell^1}=\|  \Gamma^{-1} (\rho(0))\|_{\ell^1}\lesssim \lambda^{-1} N\,2^N.
\end{equation*}
Then, using again \eqref{cond:final} and recalling also \eqref{bound:diff}, 
\[
\| r(0)-r^{\lambda}(0) \|_{\ell^1}\lesssim q^{-2} N^3 8^{N-1} \lambda^{-3}\ll \lambda^{-2}.
\]
Hence, $r(t)$ satisfies the assumption \eqref{assump:initial} and we can apply Proposition \ref{prop:approxarg} to estimate $\rho_n(T)$ for $n\in \Lambda_{N-2}$, where $T$ is in \eqref{def:timeT}. 
We split it as
\begin{align*}
|\rho_n(T)|&\geq |r_n(T)|-\left|  \Gamma_n(\{ r_n(T)\,e^{\mathrm{i} \lambda(n) T} \})(T)- r_n(T)\,e^{\mathrm{i} \lambda(n) T}\right|\\
&\geq |r^{\lambda}_n(T)|-|r_n(T)-r^{\lambda}(T)|-\left|  \Gamma_n(\{ r_n(T)\,e^{\mathrm{i} \lambda(n) T} \})(T)- r_n(T)\,e^{\mathrm{i} \lambda(n) T}\right|.
\end{align*}
Using that $T=\lambda^2 T_0$, the definitions in \eqref{lambda} and \eqref{def:rlambda} and  Theorem \ref{thm:orbit} we have
\[
|r^{\lambda}_n(T)|=\lambda^{-1} |b_{N-2}(T_0)|\geq \frac{3}{4} \lambda^{-1}.
\]
Now we give upper bounds for the last two terms. By the bound \eqref{bound:close} of Proposition \ref{prop:approxarg} we have that  $n\in \Lambda_{N-2}$ (taking $N$ large enough),
\[
|r_n(T)-r_n^{\lambda}(T)|\le \sum_{n\in \Z^2} |r_n(T)-r_n^{\lambda}(T)|\le \lambda^{-(1+\epsilon)}.
\]
We observe that this implies that
\[
|r_n(T)|\le 2 \lambda^{-1}.
\]
By estimate \eqref{bound:GammaId2} in Proposition \ref{prop:wbnf} we have
\begin{align*}
&\left|  \Gamma_n(\{ r_n(T)\,e^{\mathrm{i} \lambda(n) T} \})(T)- r_n(T)\,e^{\mathrm{i} \lambda(n) T}\right|\\
&\le \left\|  \Gamma_n(\{ r_n(T)\,e^{\mathrm{i} \lambda(n) T} \})(T)- r_n(T)\,e^{\mathrm{i} \lambda(n) T}\right\|_{\ell^1}\lesssim   q^{-2} \lambda^{-3}\lesssim \lambda^{-3}.
\end{align*}
In conclusion, using also \eqref{cond:final} and taking $N$ large enough,
we have
\[
|\rho_n(T)|\geq \frac{3}{4}\lambda^{-1}-\lambda^{-(1+\epsilon)}-C \lambda^{-3}\geq \frac{\lambda^{-1}}{2},
\]
where $C>0$ is a universal constant. This implies
\begin{equation}\label{bound:finalnorm}
\| \rho(T)\|_s^2\geq \frac{\lambda^{-2}}{4} S_{N-2}.
\end{equation}
Now we claim the following
\begin{equation}\label{bound:initialnorm}
\| \rho(0)\|_s^2\sim \lambda^{-2} S_3.
\end{equation}
Since $\rho(0)=r^{\lambda}(0)$ is supported on $\Lambda$, the initial Sobolev norm is given by
\[
\| \rho(0)\|_s^2=\sum_{n\in\Lambda} |n|^{2s} |r^{\lambda}_n(0)|^2.
\]
Then recalling also the definition of $r^{\lambda}$ in \eqref{def:rlambda} and using Theorem \ref{thm:orbit} we have
\[
\sum_{n\in\Lambda} |n|^{2s} |r^{\lambda}_n(0)|^2\le \lambda^{-2} S_3+\lambda^{-2} \delta^{2\sigma} \sum_{i\neq 3} S_i\le \lambda^{-2} S_3 \left( 1+\delta^{2\sigma} \sum_{i\neq 3} \frac{S_i}{S_3}\right).
\]
By \eqref{bound:gen}, we have that
\[
\delta^{2\sigma} \sum_{i\neq 3} \frac{S_i}{S_3}\le \delta^{2\sigma} C\,(N-1)\,e^{s N}
\]
for some universal constant $C>0$. Since $\delta=e^{-\gamma N}$ with a $\gamma$ large to be fixed (see Theorem \ref{thm:orbit}), we can take $\gamma$ such that $2\gamma\sigma>s$. Then
\[
\lambda^{-2} S_3 \left( 1+\delta^{2\sigma} \sum_{i\neq 3} \frac{S_i}{S_3}\right)\sim \lambda^{-2} S_3.
\]
On the other hand
\[
\| \rho(0)\|_s^2=\sum_{n\in\Lambda} |n|^{2s} |r^{\lambda}_n(0)|^2\geq \sum_{n\in\Lambda_3} |n|^{2s} |r^{\lambda}_n(0)|^2\geq \lambda^{-2} S_3 (1-\delta^{\sigma})^2\geq \tilde{C} \lambda^{-2} S_3
\]
for some universal constant $\tilde{C}>0$.
This concludes the proof of the claim.
By Theorem \ref{thm:gen}, \eqref{bound:initialnorm} and \eqref{bound:finalnorm} we have
\[
\frac{\| \rho(T) \|_s^2}{\| \rho(0) \|_s^2}\gtrsim \frac{S_{N-2}}{S_{3}}\gtrsim 2^{(s-1) (N-6)}.
\]
\end{proof}

\paragraph{Conclusion of the proof of Theorem \ref{thm:nosmall}}
Fix $s>1$ and $\mathcal{C}\gg 1$ .
We consider $N$ such that
\[
2^{(s-1) (N-6)}\geq {\mathcal{C}^2} \qquad \Rightarrow \qquad N\sim \frac{\log(\mathcal{C})}{s-1}.
\]
Moreover, taking for instance
\[
\lambda= \exp(5^N),
\]
 and the smaller $q$ such that
\[
q\,\psi(q)\le N^{-7} 72^{-N} \lambda^{-2(1+\epsilon)} R^{-2}
\]
 the conditions \eqref{cond:final} hold. Then the growth is given by Lemma \ref{lem:ratio} and
the estimate on the diffusion time \eqref{time1} comes from \eqref{bound:timeT0}, \eqref{def:timeT}.

\paragraph{Conclusion of the proof of Theorem \ref{thm:small}}
Fix $s>1$, $\tau>2s$, $\mathcal{C}\gg1$ and $\mu\ll 1$ and  recall that, for Theorem \ref{thm:small}, we consider $\omega$'s which are $\psi$--approximable with 
\begin{equation}\label{psitau}
\psi(q)=\frac{\tc}{q^{1+\tau}}, \qquad \tc\geq 1
\end{equation}
and satisfy
\[
 \left|\omega-\frac{p}{q}\right|\geq \frac{1}{q^{1+ \log q}}
\]
for $q$ large enough and any $p\in\mathbb{N}$. To control the initial Sobolev norm it will be crucial to control the size of the convergents of $\omega$. Given a convergent sequence $(p_n, q_n)\in \Z\times \N$ we have
\[
\frac{1}{q_n (q_{n+1}+q_n)}\le  \left|\omega-\frac{p_n}{q_n}\right|\le \frac{1}{q_n q_{n+1}}.
\]
We deduce the following lower and upper bounds for large enough denominators of the convergents $(p_n, q_n)$
\begin{equation}\label{bound:qn+1}
q_n^{1+\tau}-q_n\le q_{n+1}\le q_n^{\log q_n}.
\end{equation}
Since $\lim_{n\to +\infty} q_n= +\infty$ then, for some fixed $n>0$, there exists $\mathtt{a}, \tilde{\kappa}, \kappa\gg 1$ such that
\[
q_n\in \left[\exp(\tilde{\kappa}\mathtt{a}^N), \exp(\kappa \mathtt{a}^N)\right].
\]
By \eqref{bound:qn+1} we have that for all $m>0$
\begin{equation}\label{bound:allq2}
q_{n+m}\in \left[\exp(\tilde{\kappa}\mathtt{a}^N), \exp( \kappa^m \mathtt{a}^{mN})\right].
\end{equation}
We shall use this fact later, because a control on the size of the successive convergents gives a control on the upper bound of the instability time.

 We consider $N$ such that
\[
2^{(s-1) (N-6)}\geq \frac{\mathcal{C}^2}{\mu^2} \qquad \Rightarrow \qquad N\sim \frac{\log(\mathcal{C}/\mu)}{s-1}.
\]
Lemma \ref{lem:ratio} gives the growth of Sobolev norms. It only remains to  impose that the Sobolev norm of the initial condition is  of order $\mu$. 
{ By \eqref{bound:initialnorm} we need to ask that
\begin{equation}\label{cond:roma}
\| \rho(0)\|_s^2\sim \lambda^{-2} S_3\sim \mu^{2}.
\end{equation}
}
By \eqref{def:R} we have that
\[
C^{-2s}  2^{N-1}\,\,q^{2 s}\,R^{2s}\le S_3\le C^{2s}  2^{N-1}\,3^{2Ns}\,q^{2 s}\,R^{2s}.
\]
Thus we can take $\lambda^2\sim \mu^{-2} S_3$, which satisfies
\begin{equation}\label{lowuplambda}
C^{-s} 2^{N/2} q^s  R^s {\mu^{-1}}\lesssim \lambda\lesssim C^s 2^{N/2} 3^{Ns} q^s R^s \,\mu^{-1}.
\end{equation}
We note that we can enlarge the above interval by choosing $N$ much larger (that can be done, for instance, enlarging $\mathcal{C}$ since $N\sim \log(\mathcal{C}/\mu)$).
Recalling that $R\geq \exp(\alpha^N)$ with $\alpha\gg 1$ (see for instance Theorem \ref{thm:gen}) it is easy to see that the chosen $\lambda$ satisfies the first inequality in \eqref{cond:final} taking $N$ large enough. By the choice of $\psi$ in \eqref{psitau}, the second inequality in \eqref{cond:final} is
\[
\lambda \le R^{-\tfrac{1}{1+\epsilon}} q^{\tfrac{\tau}{2(1+\epsilon)}} (72^{-N} N^{-7})^{\tfrac{1}{2(1+\epsilon)}} .
\]
We remark that $\epsilon$ given in Proposition \ref{prop:approxarg} can be considered arbitrarily small.
The above inequality is compatible with \eqref{lowuplambda} if
\begin{equation}\label{cond:compat}
 2^{N/2} q^s  R^s {\mu^{-1}} \lesssim R^{-\tfrac{1}{1+\epsilon}} q^{\tfrac{\tau}{2(1+\epsilon)}} (72^{-N} N^{-7})^{\tfrac{1}{2(1+\epsilon)}}.
\end{equation}
Let us call 
\[
\nu:=\frac{\tau}{2(1+\epsilon)}-s.
\]
Since $\tau>2s$ and $\epsilon>0$ is arbitrarily small we have that $\nu>0$. Therefore if
\begin{equation}\label{explosion}
 \tilde{\kappa}=\frac{10 s(1+\tilde{\eta})}{\tau-2 s}
\end{equation}
where $\tilde{\eta}$ is the constant introduced in Theorem \ref{thm:gen}.

We have that, for all $q\geq \exp(\tilde{\kappa} \alpha^N)$ and $N$ large enough,
\begin{equation}\label{lowbound:q}
q^{\nu}\gtrsim 2^{N/2}   R^{s+\tfrac{1}{1+\epsilon}}\,(72^N N^7)^{\tfrac{1}{2(1+\epsilon)}} {\mu^{-1}},
\end{equation}
which is the compatibility condition \eqref{cond:compat}. By the discussion above (see \eqref{bound:allq2}) there exist $\kappa\gg \tilde{\kappa}$ and $\sigma\gg 1$ such that the interval $[\exp(\tilde{\kappa} \alpha^N), \exp({\kappa} \alpha^{\sigma N})]$ contains a convergent $q\in\mathbb{N}$ that satisfies \eqref{lowbound:q}.

Then, one can chose   $\lambda$ in the interval \eqref{lowuplambda} such that the initial Sobolev norm satisfies \eqref{cond:roma}. That is, 
\[
\tilde{C}^{-1}\mu \le \| \rho(0) \|_s\le \tilde{C} \mu
\]
for some $\tilde{C}>0$ independent of $\mu$. Moreover,
the chosen $\lambda$ and $q$ also satisfy \eqref{cond:final} and \eqref{assumption} respectively. Therefore, we can apply Proposition \ref{prop:approxarg}.

The upper bound \eqref{time2} on the time $T$ in \eqref{def:timeT} is obtained by using the bounds \eqref{lowuplambda}, the fact that $q\le \exp(\kappa\alpha^{\sigma N})$ and the upper bound of $R$ in Theorem \ref{thm:gen}. We remark that by \eqref{explosion} if $\tau\to 2s^+$ the constant $\tilde{\kappa}$, and so $\kappa$, tends to $+\infty$. This means that we lose completely the control on the time $T$ when $\tau$ approaches the lower bound $2s$.

\paragraph{Case $s\in (0, 1)$.} The proofs of Theorems \ref{thm:nosmall} and \ref{thm:small} with $s\in (0, 1)$ follow the same lines as the ones for $s>1$. For full details we refer to \cite{GuardiaHHMP19}. The main difference is that we need to consider initial data supported on ${S}_{N-2}$. The energy moves from $S_{N-2}$ to $S_3$, hence we get
\[
\frac{\| \rho(T) \|_s^2}{\| \rho(0) \|_s^2}\gtrsim \frac{S_{3}}{S_{N-2}}\gtrsim 2^{(1-s) (N-6)}.
\]

\bibliographystyle{plain}
\bibliography{biblio}

\def\cprime{$'$} \def\cprime{$'$}
\begin{thebibliography}{10}

\bibitem{BGMR18}
D.~Bambusi, B.~Gr{\'e}bert, A.~Maspero, and R.~Didier.
\newblock Reducibility of the quantum harmonic oscillator in $d$-dimensions
  with polynomial time-dependent perturbation.
\newblock {\em Anal. PDE}, 11(3):775--799, 2018.

\bibitem{BRV}
V.~Beresnevich, F.~Ram\'irez, and S.~Velani.
\newblock {\em Metric Diophantine Approximation: Aspects of Recent Work}.
\newblock London Mathematical Society Lecture Note Series. Cambridge University
  Press, 2016.

\bibitem{BFGI}
J.~Bernier, R.~Feola, B.~Gr{\'e}bert, and F.~Iandoli.
\newblock Long-time existence for semi-linear beam equations on irrational
  tori.
\newblock {\em Journal of Dynamics and Differential Equations},
  33(3):1363--1398, 2021.

\bibitem{BMflat}
M.~Berti and A.~Maspero.
\newblock Long time dynamics of schr{\"o}dinger and wave equations on flat
  tori.
\newblock {\em Journal of Differential Equations}, 267(2):1167--1200, 2019.

\bibitem{Besi}
A.~S. Besicovitch.
\newblock Sets of fractional dimension ($iv$): On rational approximation to
  real numbers.
\newblock {\em J. Math. London Math. Soc.}, 9:126--131, 1934.

\bibitem{Bou95}
J.~Bourgain.
\newblock Aspects of long time behaviour of solutions of nonlinear
  {H}amiltonian evolution equations.
\newblock {\em Geom. Funct. Anal.}, 5(2):105--140, 1995.

\bibitem{Bourgain96}
J.~Bourgain.
\newblock On the growth in time of higher {S}obolev norms of smooth solutions
  of {H}amiltonian {PDE}.
\newblock {\em Internat. Math. Res. Notices}, 6:277--304, 1996.

\bibitem{Bourgain99}
J.~Bourgain.
\newblock Growth of {S}obolev norms in linear {S}chr{\"o}dinger equations with
  quasi-periodic potential.
\newblock {\em Comm. Math. Phys.}, 204(1):207--247, 1999.

\bibitem{Bourgain00b}
J.~Bourgain.
\newblock Problems in {H}amiltonian {PDE}'s.
\newblock {\em Geom. Funct. Anal.}, Special Volume, Part I:32--56, 2000.
\newblock GAFA 2000 (Tel Aviv, 1999).

\bibitem{CatoireW10}
F.~Catoire and W.-M. Wang.
\newblock Bounds on {S}obolev norms for the defocusing nonlinear
  {S}chr\"odinger equation on general flat tori.
\newblock {\em Commun. Pure Appl. Anal.}, 9(2):483--491, 2010.

\bibitem{CKSTT}
J.~Colliander, M.~Keel, G.~Staffilani, H.~Takaoka, and T.~Tao.
\newblock Transfer of energy to high frequencies in the cubic defocusing
  nonlinear {S}chr{\"o}dinger equation.
\newblock {\em Invent. Math.}, 181(1):39--113, 2010.

\bibitem{CollianderDKS01}
J.~E. Colliander, J.-M. Delort, C.~E. Kenig, and G.~Staffilani.
\newblock Bilinear estimates and applications to 2{D} {NLS}.
\newblock {\em Trans. Amer. Math. Soc.}, 353(8):3307--3325 (electronic), 2001.

\bibitem{BLM}
B.~Langella D.~Bambusi and R.~Montalto.
\newblock Growth of sobolev norms for unbounded perturbations of the laplacian
  on flat tori.
\newblock {\em Preprint available at https://arxiv.org/abs/2012.02654}, 2021.

\bibitem{Delort2014}
J.-M. Delort.
\newblock Growth of sobolev norms for solutions of time dependent
  schr{\"o}dinger operators with harmonic oscillator potential.
\newblock {\em Communications in Partial Differential Equations}, 39(1):1--33,
  2014.

\bibitem{Deng}
Y.~Deng.
\newblock On growth of sobolev norms for energy critical nls on irrational
  tori: Small energy case.
\newblock {\em Communications on Pure and Applied Mathematics}, 72:801--834,
  2019.

\bibitem{DengGerm}
Y.~Deng and P.~Germain.
\newblock Growth of solutions to {NLS} on irrational tori.
\newblock {\em International Mathematics Research Notices}, 2019(9):2919--2950,
  09 2017.

\bibitem{FaouRaphael}
P.~Raphael E.~Faou.
\newblock On weakly turbulent solutions to the perturbed linear harmonic
  oscillator.
\newblock {\em Preprint available at https://arxiv.org/abs/2006.08206}, 2020.

\bibitem{FIM}
R.~Feola, F.~Iandoli, and F.~Murgante.
\newblock Long-time stability of the quantum hydrodynamic system on irrational
  tori.
\newblock {\em Mathematics in Engeenering}, 4, 2021.

\bibitem{FeolaMont}
R.~Feola and R.~Montalto.
\newblock Quadratic lifespan and growth of {S}obolev norms for derivative
  {S}chr{\"o}dinger equations on generic tori.
\newblock {\em Journal of Differential Equations}, 312:276--316, 2022.

\bibitem{GerardG10}
P.~G{\'e}rard and S.~Grellier.
\newblock The cubic {S}zeg{\"o} equation.
\newblock {\em Ann. Sci. \'Ec. Norm. Sup\'er. (4)}, 43(5):761--810, 2010.

\bibitem{GerardG11}
P.~G{\'e}rard and S.~Grellier.
\newblock Effective integrable dynamics for a certain nonlinear wave equation.
\newblock {\em Anal. PDE}, 5(5):1139--1155, 2012.

\bibitem{GiulianiDC}
F.~Giuliani.
\newblock Transfers of energy through fast diffusion channels in some resonant
  pdes on the circle.
\newblock {\em Discrete and Continuous Dynamical Systems}, 41(11):5057--5085,
  2021.

\bibitem{GGMP}
F.~Giuliani, M.~Guardia, P.~Martin, and S.~Pasquali.
\newblock Chaotic-like transfers of energy in hamiltonian pdes.
\newblock {\em Comm. Math. Phys.}, 384(2):1227--1290, 2021.

\bibitem{GGMPr}
F.~Giuliani, M.~Guardia, P.~Martin, and S.~Pasquali.
\newblock Chaotic resonant dynamics and exchanges of energy in hamiltonian
  pdes.
\newblock {\em Atti Accad. Naz. Lincei Cl. Sci. Fis. Mat. Natur.}, 32:149--166,
  2021.

\bibitem{GPT}
B.~Gr{\'e}bert, {\'E}.~Paturel, and L.~Thomann.
\newblock Beating effects in cubic {S}chr\"odinger systems and growth of
  {S}obolev norms.
\newblock {\em Nonlinearity}, 26(5):1361--1376, 2013.

\bibitem{GT}
B.~Gr{\'e}bert and L.~Thomann.
\newblock Resonant dynamics for the quintic nonlinear {S}chr\"odinger equation.
\newblock {\em Ann. Inst. H. Poincar\'e Anal. Non Lin\'eaire}, 29(3):455--477,
  2012.

\bibitem{Guardia14}
M.~Guardia.
\newblock Growth of {S}obolev norms in the cubic nonlinear {S}chr\"odinger
  equation with a convolution potential.
\newblock {\em Comm. Math. Phys.}, 329(1):405--434, 2014.

\bibitem{GuardiaHHMP19}
M.~Guardia, E.~Haus, Z.~Hani, A~Maspero, and M.~Procesi.
\newblock Strong nonlinear instability and growth of {S}obolev norms near
  quasiperiodic finite-gap tori for the 2{D} cubic {N}{L}{S} equation.
\newblock To appear on J. Eur. Math. Soc. (JEMS), 2020.

\bibitem{GuardiaHP16}
M.~Guardia, E.~Haus, and M.~Procesi.
\newblock Growth of {S}obolev norms for the analytic {NLS} on {$\Bbb{T}^2$}.
\newblock {\em Adv. Math.}, 301:615--692, 2016.

\bibitem{GuardiaK12}
M.~Guardia and V.~Kaloshin.
\newblock Growth of {S}obolev norms in the cubic defocusing nonlinear
  {S}chr\"{o}dinger equation.
\newblock {\em J. Eur. Math. Soc. (JEMS)}, 17(1):71--149, 2015.

\bibitem{GuardiaK12Err}
M.~Guardia and V.~Kaloshin.
\newblock Erratum to ``{G}rowth of {S}obolev norms in the cubic defocusing
  nonlinear {S}chr\"odinger equation'' [{MR}3312404].
\newblock {\em J. Eur. Math. Soc. (JEMS)}, 19(2):601--602, 2017.

\bibitem{Hani12}
Z.~Hani.
\newblock Long-time instability and unbounded {S}obolev orbits for some
  periodic nonlinear {S}chr\"odinger equations.
\newblock {\em Arch. Ration. Mech. Anal.}, 211(3):929--964, 2014.

\bibitem{HaniV2015}
Z.~Hani, B.~Pausader, N.~Tzvetkov, and N.~Visciglia.
\newblock Modified scattering for the cubic {S}chr{\"o}dinger equation on
  product spaces and applications.
\newblock {\em Forum of Mathematics, Pi}, 3, 2015.

\bibitem{HPquintic}
E.~Haus and M.~Procesi.
\newblock Growth of {S}obolev norms for the quintic {NLS} on {$T^2$}.
\newblock {\em Anal. PDE}, 8(4):883--922, 2015.

\bibitem{HausP17}
E.~Haus and M.~Procesi.
\newblock K{AM} for beating solutions of the quintic {NLS}.
\newblock {\em Comm. Math. Phys.}, 354(3):1101--1132, 2017.

\bibitem{StafWil2}
A.~Hrabsky, Y.~Pan, G.~Staffilani, and B.~Wilson.
\newblock Energy transfer for solutions to the nonlinear {S}chr\"odinger
  equation on irrational tori.
\newblock {\em Preprint available at https://arxiv.org/abs/2107.01459}, 2021.

\bibitem{Jarnik}
V.~Jarn\'ik.
\newblock Diophantische approximationen und hausdorffsches mass.
\newblock {\em Mat. Sb.}, 36:371--382, 1929.

\bibitem{Kinch}
A.~Khintchine.
\newblock Einige s\"atze \"uber kettenbr\"uche, mit anwendungen auf die theorie
  der diophantischen approximationen.
\newblock {\em Math. Ann.}, 92:115--125, 1924.

\bibitem{Kuksin97}
S.~B. Kuksin.
\newblock On turbulence in nonlinear {S}chr{\"o}dinger equations.
\newblock {\em Geom. Funct. Anal.}, 7(4):783--822, 1997.

\bibitem{Kuksin96}
S.B. Kuksin.
\newblock Growth and oscillations of solutions of nonlinear {S}chr{\"o}dinger
  equation.
\newblock {\em Comm. Math. Phys.}, 178(2):265--280, 1996.

\bibitem{Kuksin97b}
S.B. Kuksin.
\newblock Oscillations in space-periodic nonlinear {S}chr{\"o}dinger equations.
\newblock {\em Geom. Funct. Anal.}, 7(2):338--363, 1997.

\bibitem{Maspero18g}
A.~{Maspero}.
\newblock Lower bounds on the growth of sobolev norms in some linear time
  dependent schr\"odinger equations.
\newblock {\em Math. Res. Lett.}, 26(4), 2019.

\bibitem{MasperoDispersive}
A.~Maspero.
\newblock Growth of sobolev norms in linear schr\"odinger equations as a
  dispersive phenomenon.
\newblock {\em Preprint available at https://arxiv.org/abs/2101.09055}, 2021.

\bibitem{PTV17}
F.~Planchon, N.~Tzvetkov, and N.~Visciglia.
\newblock On the growth of {S}obolev norms for {NLS} on 2- and 3-dimensional
  manifolds.
\newblock {\em Anal. PDE}, 10(5):1123--1147, 2017.

\bibitem{Sohinger11}
V.~Sohinger.
\newblock Bounds on the growth of high {S}obolev norms of solutions to
  nonlinear {S}chr\"odinger equations on {$S^1$}.
\newblock {\em Differential Integral Equations}, 24(7-8):653--718, 2011.

\bibitem{Staffilani97}
G.~Staffilani.
\newblock Quadratic forms for a {$2$}-{D} semilinear {S}chr\"odinger equation.
\newblock {\em Duke Math. J.}, 86(1):79--107, 1997.

\bibitem{StafWil}
G.~Staffilani and B.~Wilson.
\newblock Stability of the {C}ubic nonlinear {S}chrodinger {E}quation on an
  irrational torus.
\newblock {\em SIAM J. Math. Anal.}, 52:1318--1342, 2020.

\bibitem{Thomann2020GrowthOS}
L.~Thomann.
\newblock Growth of sobolev norms for linear schr\{\"o\}dinger operators.
\newblock {\em Annales Henri Lebesgue}, (4):1595--1618, 2020.

\bibitem{Liang}
Z.~Zhao Z.~Liang and Q.~Zhou.
\newblock $1$-d quantum harmonic oscillator with time quasi-periodic quadratic
  perturbation: reducibility and growth of sobolev norms.
\newblock 146(1):158--182, 2021.

\end{thebibliography}

\end{document}